\documentclass[12pt]{article}
\usepackage[left=1in,top=1in,right=1in,bottom=1in,nohead]{geometry}

\usepackage{url,subfig,latexsym,amsmath,amssymb, amsfonts,enumerate,
  epsfig, graphics,times, amsthm,mathtools}
\usepackage{enumitem}
\usepackage{bbm}
\usepackage{color}
\usepackage{nameref}
\usepackage[colorlinks]{hyperref}
\usepackage[export]{adjustbox}
\usepackage{wrapfig}
\usepackage{thmtools}

\newtheorem{theorem}{Theorem}[section]

\newtheorem{lemma}[theorem]{Lemma}

\newtheorem*{thm*}{Theorem}

\newtheorem{definition}{Definition}[section]
\newtheorem{remark}{Remark}[section]
\newcommand{\Zs}{Z^{N,\gamma}}
\newcommand{\Zt}{Z^{N,-\theta_0}}

\theoremstyle{definition}
\newtheorem{example}{Example}[section]

\newcommand{\mybox}[1]{%
  \setbox0=\hbox{#1}%
  \setlength{\@tempdima}{\dimexpr\wd0+13pt}%
  \begin{tcolorbox}[colframe=mycolor,boxrule=0.5pt,arc=4pt,
      left=6pt,right=6pt,top=6pt,bottom=6pt,boxsep=0pt,width=\@tempdima]
    #1
  \end{tcolorbox}
}

\newcommand{\blue}{\textcolor{black}}

\newcommand{\dlim}{\displaystyle \lim\limits}
\newcommand{\dsum}{\displaystyle \sum\limits}
\newcommand{\dint}{\displaystyle \int}

\def\S{\mathcal S}
\def\C{\mathcal C}
\def\Re{\mathcal R}
\def\R{\mathbb R}

\def\Z{\mathbb Z}

\def\K{\mathcal K}
\def\Td1{T^{D,1}_{\{x_n\}}}
\def\Ts1{T^{S,1}_{\{x_n\}}}

\def\K{\mathcal{K}}

\def\Tid1{T^{D,1}_{\{x_n+\zeta_i\}}}
\def\Tis1{T^{S,1}_{\{x_n+\zeta_i\}}}
\def\Tjs1{T^{S,1}_{\{x_n+\zeta_j\}}}
\def\Tjd1{T^{D,1}_{\{x_n+\zeta_j\}}}

\newcount\colveccount
\newcommand*\colvec[1]{
        \global\colveccount#1
        \begin{pmatrix}
        \colvecnext
}
\def\colvecnext#1{
        #1
        \global\advance\colveccount-1
        \ifnum\colveccount>0
                \\
                \expandafter\colvecnext
        \else
                \end{pmatrix}
        \fi
}

\title{Accuracy of Multiscale Reduction for Stochastic Reaction Systems}

\author{
German Enciso\thanks{Department of Mathematics, University of
  California, Irvine, USA.  enciso@uci.edu},
\and
Jinsu Kim\thanks{Department of Mathematics, University of
  California, Irvine, USA.  jinsu.kim@uci.edu}
  }

\begin{document}

\maketitle

\begin{abstract}
\noindent 
Stochastic models of chemical reaction networks are an important tool to describe and analyze noise effects in cell biology. When chemical species and reaction rates in a reaction system have different orders of magnitude, the associated stochastic system is often modeled in a multiscale regime. 
It is known that multiscale models can be approximated with a reduced system such as mean field dynamics or hybrid systems, but the accuracy of the approximation remains unknown.
In this paper, we estimate the probability distribution of low copy species in multiscale stochastic reaction systems under short-time scale.  We also establish an error bound for this approximation.  Throughout the manuscript, typical mass action systems are mainly handled, but we also show that the main theorem can extended to general kinetics, which generalizes existing results in the literature.  Our approach is based on a direct analysis of the Kolmogorov equation, in contrast to classical approaches in the existing literature.

\end{abstract}

\section{Introduction}

%Intro Structure?
%
%*  intro to stochastic reaction networks and their solutions
%     Write it so that someone without a background in stochastic systems can understand
%
%*  multiscaling
%    - describe in more (intuitive) detail what order 1 multiscaling means
%
%*  SIR example
%
%*  main results
%     - what exactly are you proving?  Convergence to stochastic system, and 
%      |p^N(t) - p(t)| is O(N^c)
%
%*  outline of sections

Consider a reaction network with a discrete number of copies for each species, a problem of increasing importance in cell biology.  The evolution of the copy number of each chemical species can be modeled using a continuous-time, discrete-space Markov process, and stochastic effects are well known to be present. The time evolution of this system can be computed by solving the so-called Kolmogorov equation, but this direct approach is rarely useful because of its high dimensionality. Therefore computational and analytic methods for estimating the distribution associated with a stochastic reaction network system have been developed \cite{ball2006asymptotic, cao2008optimal,  cao2016accurate, Gupta_Khammash2017, StoQSSA2019, KangKurtz2013,  kim2017reduction, sontag2017reduction,  Munskythesis, MunskyKhammash2008}.

%be more specific what you mean with your multiscaling.  N = total protein concentration, some proteins O(N), other proteins O(1). 
%To explain constant order, suppose A is O(N), B is O(1), then for a reaction 
%A+B->C
%we set reaction parameter k to be O(N^{-1}), so that the reaction rate k*A*B is O(1).   

In a stochastic system, some chemical species can have much higher molecular counts that the others. Furthermore, the intensity of each reaction can also vary over several orders of magnitude. For example, suppose $A$ and $B$ are proteins interacting in the network system 
\blue{\begin{align}\label{eq:example in intro}
A +B \xrightleftharpoons[\kappa_2 N]{\kappa_1} 2A \xleftarrow{\kappa_3} C, 
\end{align}
with the initial conditions $A(0)=1$, $B(0)=N$ and $C(0)=10$.  Here $N$ is a large scaling parameter, and it could mean the Avogadro number $6\times 10^{23}$, the total mass of the system, or the volume of the container where the reaction takes place.
Letting $X(t)=(A(t),B(t),C(t))$ be the stochastic process associated with the reaction network at time $t$, we suppose that the intensities of the reactions $A+B\to 2A$, $2A\to A+B$ and $C\to 2A$ are $\lambda_{A+B\to 2A}(X(t))=\kappa_1 A(t)B(t)$, $\lambda_{2A\to A+B}(X(t))=\kappa_2 N A(t)(A(t)-1)$ and $\lambda_{C\to 2A}(X(t))=\kappa_3 C(t)$, respectively. Note that around $t=0$, the intensity of $C\to \emptyset$ is much smaller than the intensities of the other reactions. Thus we can expect that reaction $C\to 2A$ is rarely fired and hence the copy number of $C$ evolves in slow-time scale.}

As shown in the example above, due to the size of intensities, the dynamics of each species in a reaction system can have different time scales. If the size of species and the size of intensities satisfy a particular balance condition, then the system can be decomposed into two or more subsystems, each of which converges to a lower dimensional system as the scaling parameter tends to infinity \cite{ball2006asymptotic, KangKurtz2013}. Depending on the time scale, the limiting system of the subsystems could be either a stochastic, deterministic or piece-wise deterministic model \cite{anderson2017finite, ball2006asymptotic,  enciso2019embracing, StoQSSA2019, KangKurtz2013}.

%  One special scaling condition we consider in this manuscript is that each reaction intensity is of constant order, so that all species change at a slow time scale.  This type of scaling can be seen for instance in ecology or epidemiology models such as the SIR model \cite{hethcote2000mathematics},

% \begin{align*}
% S+I\xrightarrow{\kappa_1/N} 2I, \quad I \xrightarrow{\kappa_2}  R,
% \end{align*} 
%   where an infectious group $I$ with initial low population has limited ability to contact a susceptible group $S$ with high population as shown in Figure \ref{fig:intro}.  
  
% \begin{wrapfigure}{I}{0.3\textwidth}
% 	\centering
% 	{\includegraphics{intro.pdf}}
% \caption{SIR model with constant order intensities}\label{fig:intro} 
% \end{wrapfigure}  

%  We assume that $S(0)=N$ with $N$ large, and that $I(0)$ is a small constant.  Since each infectious individual encounters a limited number of susceptible individuals, we assume that the intensity of the reaction $S+I\to 2I$ is of constant order $\lambda_{S+I\to 2I}=\kappa_1 I S /N$.  Similarly $\lambda_{I\to R}=\kappa_2 I$.    
 
 %first discuss under what circumstances there is a limit system

\blue{In this paper, we show that a multiscale stochastic model under a short-term timescale converges in distribution to an associated reduced model. To illustrate this, we consider the toy model in \eqref{eq:example in intro}. By modeling system \eqref{eq:example in intro} with a continuous time Markov process, the holding time for the next reaction is exponentially distributed with rate $\lambda_{A+B\to 2A}(X(t))+\lambda_{2A\to A+B}(X(t))+\lambda_{C\to 2A}(X(t))$. Hence the expected holding time until the next reaction around $t=0$ is of order $\frac{1}{N}$. This indicates that the number of reactions fired within $[0,\frac{T}{N}]$ is of constant order on average for $T>0$. This amount of transitions is substantial for the low copy species $A$ and $C$, but it is too small to considerably fluctuate the concentration of $B$, $\frac{B(t)}{N}$, within $[0,\frac{T}{N}]$. In this vein, for the scaled process $(A(t)),\frac{B(t)}{N},C(t))$, we can approximately freeze $\frac{B(t)}{N}$ at $\frac{B(0)}{N}=1$. Then the dynamics of species $A(t/N)$ in the original system \eqref{eq:example in intro} under the short-term timescale can be approximated with a reduced system
$A\xrightleftharpoons[\kappa_2 ]{\kappa_1 } 2A$.}

In general, we show that for some $\theta_0>0$, the short-term distribution $p^N(t/N^{\theta_0})$ of the original network system converges to the distribution $p(t)$ of the limiting reduced system for any $t$ in a compact time interval $[0,T]$, as the scaling parameter $N$ tends to infinity.  The main theoretic tools for this result rely on a direct calculation using the Kolmogorov equation, and this allows us to obtain the error bound 
\begin{align}\label{eq:error bound intro}
\sup_{t\in [0,T]}|p^N(U,t/N^{\theta_0})-p(U,t)|\le \frac{c}{N^\nu}. 
\end{align}
Here $U$ is an arbitrary subset of the state space, and the constants $c>0$ and $\nu \in (0,1)$ are independent of $U$.

For the main theorem and  relevant lemmas, we assume that the associated stochastic system for a reaction system is under mass-action kinetics. However, the main result can be extended to general kinetics such as Michealis-Mentum kinetics and hill type functions. Furthermore since the proof of the main theorem does not require network structural restrictions, this convergence result holds not only for bimolecular chemical reaction systems, but also for higher order reaction systems with general kinetics. Instead of structural restrictions, we assume that the reduced reaction system admits a stationary distribution with a finite moment condition. Since the finite moment condition of the reduced network system guarantees the non-explosivity of the original multiscale model. Therefore the error bound \eqref{eq:error bound intro} holds for any time $T$. By showing this error bound, this work provides a addition to previous related studies of multiscaling limits. 

%The results in this work were inspired by the problem of describing the stationary distribution of stochastic, absolutely robust chemical reaction networks \cite{Enciso2016}. We have used the error bounds established here in order to approximate the marginal distribution of absolutely robust species \cite{kim2020absolutely}.

This manuscript has the following outline. In Section \ref{sec:reactionnetworks} we introduce the basic notions of a stochastic system for a reaction network. In the same section, we also provide a multiscale framework for the stochastic model and introduce the idea of network projection. Key lemmas and the main theorem of this paper including proofs are introduced in Section \ref{sec:main}. In section \ref{sec:generalization}, we discuss some extension of the main result for general kinetics. In Section \ref{sec:examples}, in order to demonstrate how the main theorem can be applied for practical models, we provide various examples of biological models such as a futile cycle, a system of yeast polarization, p53 response to DNA damage and a population model with three species. In Appendix \hyperref[app:table]{A}, a table of notations used in the entire paper is provided. In Appendix \hyperref[app1]{B} proofs of some technical lemmas are given.

\section{Preliminaries} \label{sec:reactionnetworks}

\subsection{Stochastic Reaction Networks}

In this section, we provide a mathematical description of chemical reaction networks, with an emphasis on their associated stochastic dynamics. A \emph{reaction network} is a graphical configuration consisting of constituent species, complexes (that is, combinations of species), and reactions between complexes. A triple $(\S,\C,\Re)$ represents a reaction network where $\S, \C$ and $\Re$ are collections of species, complexes and reactions, respectively. 

\begin{definition}
A reaction network is defined with a triple of finite sets $(\S,\C,\Re)$ such that
\begin{enumerate}
\item the species set $\S=\{S_1,S_2,\dots,S_d\}$ contains the species belonging to the reaction network,
\item the complex set $\C=\{y_1,y_2,\dots,y_c\}$ contains complexes $y_k$, where for each $k$,
\begin{align*}
y_k=\sum_{i=1}^d y_{ki}S_i \quad \text{for some non-negative integers $y_{ik}$, and} 
\end{align*}
\item the reaction set $\Re=\{R_1,R_2,\dots,R_r\}$ consists of ordered pairs $(y,y')$ such that $y, y'\in \C$.
\end{enumerate}
\end{definition}

In the graphical configuration of a reaction network $(\S,\C,\Re)$, we represent complexes $y,y’,\ldots \in \C$ by nodes, and we use directed edges $y \to y'$ to denote reactions $(y,y') \in \Re$. In order to define a dynamical system associated with a reaction network $(\S,\C,\Re)$, we denote a complex $y_k$ by either a linear combination of species such as $\sum_{i=1}^d y_{ki}S_i$ or a $d$-dimensional vector $y_k=(y_{k1},\dots,y_{kd})^T$, interchangeably.
In case $y_{ki}=0$ for all $i$, the corresponding complex $y_k=\sum_{i=1}^d y_{ki}S_i$ is denoted by $\emptyset$ in the graphical configuration of the reaction network.

\begin{example}\textnormal{
Consider the following reaction network describing a substrate-enzyme system with a protein dilution:
\[
	S+E \rightleftharpoons C \rightarrow E+P, \quad P\to \emptyset.
\]
For this reaction network, $\S = \{S,E,C,P\}$, $\C=\{S+E,C,E+P,\emptyset\}$ and $\Re=\{S+E\rightarrow C, C\rightarrow S+E,C\rightarrow E+P, P\to \emptyset\}$. \hfill $\triangle$}
\end{example}

We now describe the stochastic dynamics of a reaction network using a continuous time, discrete state Markov process.  At any time $t$, the counts of each species are given by a $d$-dimensional vector $X(t)=(X_1(t),\dots,X_d(t)) \in \Z^d_{\ge 0}$. The transitions of the Markov process are determined by the given reactions. In order to define the transition probabilities, we use state-dependent intensity (or propensity) functions $\lambda_{k}:\Z^d_{\ge 0} \to \R_{\ge 0}$ of the reaction $y_k\to y'_k$.  For example, the reaction $y_k \to y_k’$ induces a transition from a given state $z$ into the state $z+y_k’-y_k$ with intensity $\lambda_k(z)$.
More generally, we have

\begin{align}
P(X(t+\Delta t) = z+y'-y \ | \ X(t)=z) = \sum_{\substack{y_k\to y'_k\in \Re\\ y'_k-y_k=y'-y}}\lambda_{k}(z)  \Delta t + o(\Delta t), \label{prob}
\end{align} 
for each state $z$ in the state space $\mathbb{S}$ of the associated Markov process $X$.  The copy number of a species $S$ at time $t$ will be denoted by $S(t)$.
 Let $p(z,t)=P(X(t)=z)$, for a given state $z$ and time $t$. Then $p(z,t)$ solves the so-called Kolmogorov forward equation, also known as the chemical master equation:
\begin{align}\label{eq:master}
\frac{d}{dt}p(z,t)=\sum_{k}\lambda_{k}(z-y'_k+y_k)p(z-y'_k+y_k,t)-\sum_{k}\lambda_{k}(z)p(z,t),
\end{align}
 where $\sum_{k}$ denotes the sum over all reactions in $\Re$. A \emph{stationary distribution} $\pi$ is a positive stationary solution of the Kolmogorov equation above such that $\sum_{z}\pi(z)=1$.  

The usual choice of intensity $\lambda_k(x)$ of a reaction $y_k\to y'_k$ in a network $(\S,\C,\Re)$ with rate constant $\kappa_k$ is
\begin{align}\label{eq:mass}
\lambda_{k}(x)= \kappa_{k} x^{(y)},
\end{align}
where $n^{(k)}=n\cdot(n-1)\cdots(n-k+1)$ for non-negative integer vector $n,k$ , $n^{(k)}=0$ if $n<k$, and $u^{(v)}=\prod_{i=1}^d  u_i^{(v_i)}$ for $u \in \Z^d, v \in \Z^d_{\ge 0}$. This choice of intensities is called \textit{stochastic mass action kinetics}. \blue{An analogous \emph{deterministic mass action kinetics} for a reaction $y_k \to y_k$ is $x^{y_k}$, where we define $u^v=\prod_{i=1}^d u_i^{v_i}$ for $u, v\in\R^d_{\ge 0}$.
}  

The stochastic process $X(t)$ also has another representation, so-called random time change representation. \cite{Kurtz72} 
\begin{align}\label{eq:kurtz rep}
X(t)=X(0)+\sum_{k}Y_k\left (\int_0^t \lambda_k(X(s)) ds \right) (y'_k-y_k),
\end{align}
where $Y_k$'s are independent unit Poisson random variables.

  %[** define p(A,t)=P(X(t)\in A) … wherever this is first needed]

   %  this might be good to mention
   %Trajectories of this model are typically simulated via the Gillespie algorithm \cite{Gill76,Gill77} or the next reaction method \cite{Anderson2007a,Gibson2000}, or are approximated via tau-leaping \cite{Anderson2007b,Gill2001}.

%For a stochastic model associated with a reaction network $(\S,\C,\Re)$, if the dynamics starts at $x$, we term $x+span\{y'-y:y\to y' \in \Re\}$ a stoichiometry class. 

In the graph associated to a reaction network, the rate constants $\kappa$ typically appear next to the reaction arrow as in $y\xrightarrow{\kappa} y'$. \blue{Through this manuscript, we model a reaction system using mass-action kinetics so that each reaction intensity $\lambda_k=\lambda_k(x)=\kappa_k x^{(y_k)}, k=1,\ldots r$. For a given reaction network $(\S,\C,\Re)$, we denote by $\K$ the set of reaction intensities $\lambda_k$. We simply denote by $(\S,\C,\Re,\K)$ the system associated with the reaction network $(\S,\C,\Re)$ with $\K$ and call it a reaction system. Using this framework, the probabilities \eqref{prob} describe the dynamics of the stochastic process $X(t)$ associated with the reaction system $(\S,\C,\Re,\K)$.}

%Recalling the discussion below Definition \ref{def:21}, we will sometimes write $x^{\sum_{i=1}^d y_iS_i}$ for $\prod_{i=1}^d x_i^{y_i}$.   For example, we have $x^{2S_1 + S_2} = x_1^2 x_2$.

%For a dynamical system of $(\S,\C,\Re)$, we incorporate a rate constant $\kappa>0$ for each reaction $y\to y' \in \Re$ and denote $y \xrightarrow{\kappa} y'$. For a fixed parameter set of rate constants, \begin{align*}
%\K = \{\kappa_{y\to y'} >0 : y\to y'\in \Re\},
%\end{align*}
%we denote $(\S,\C,\Re,\K)$ a reaction network incorporated with the parameter set $\K$. 

\subsection{Multiscaling for Reaction Networks}\label{subsec:multi}

In this section, we describe how to carry out a multiscaling procedure for a given stochastic reaction system. We use a similar notation as in the work by Kang and Kurtz \cite{KangKurtz2013}. Let $N$ be a scaling parameter, which could be interpreted as either the volume of the system, Avogadro's number, or any biological parameter. \blue{We use the conventional big $\Theta$ notion: for a real-valued sequence $a_n$ 
\begin{align*}
    a_n=\Theta(n^\gamma) \text{ if and only if there exists $c>0$ such that $\dfrac{1}{c}\le \dfrac{|a_n|}{n^\gamma} \le c$ for all $n$.}
\end{align*}} \blue{Let \blue{$X^N(t)=(X^N_1(t),\dots,X^N_{d}(t))$} be a stochastic process associated with $(\S,\C,\Re,\K)$. Assuming that each species may have a different magnitude of initial abundance, we scale $X^N(t)$ by using two sets of scaling exponents,}
\begin{align}\label{eq:scaling expo}
   \{\alpha_i : S_i \in \S\}\quad \text{and}  \quad \{\beta_k : y_k\to y'_k\in \Re\}.
\end{align}
 Each $\alpha_i$ represents the size of the abundance of species $S_i$ such that $X^N  _i(0)=\Theta(N^{\alpha_i})$.
If $\alpha_i=1$, then $N^{-\alpha_i}X_i(t)$ may represent the concentration of $S_i$ at time $t$. 
\blue{For simplicity, we assume that $X^N_i(0)=N^{\alpha_i}z^0_i$ for some $z^0_i\in \R_{\ge 0}$.}
We also assume that the rate constant of the reactions have different orders of magnitude so that we have scaled rate constants $N^{\beta_k}\kappa_k$ for each reaction $y_k\to y'_k \in \Re$. 

%Let $\alpha=(\alpha_1,\dots,\alpha_d)$.
By the representation \eqref{eq:kurtz rep}, the scaled process $Z^N_i(t) =N^{-\alpha_i}X_i(t)$, $i=1,2,\dots,d$, solves
\begin{align}\label{eq:multi rep}
Z^N_i(t)&=Z^N_i(0)+\sum_k Y_k\left ( \int_0^t N^{\beta_k}\lambda_k(X^N(s))ds \right)\frac{(y'_{k,i}-y_{k,i})}{N^{\alpha_i}},
%&\text{\blue{$=Z^N_i(0)+\sum_k Y_k\left ( \int_0^t \lambda^N_k(Z^N(s))ds \right)\frac{(y'_{k,i}-y_{k,i})}{N^{\alpha_i}}$ }}
\end{align}
where $Y_k$'s are i.i.d. unit Poisson random variables.

For additional details on the definition and uses of multiscaling in stochastic and deterministic systems, see  \cite{ball2006asymptotic, enciso2019embracing, StoQSSA2019, KangKurtz2013, pfaffelhuber2015scaling}.

\subsection{Order of reaction intensities under a short-term timescale}\label{subsec:two scales}
\blue{
In this section we show that any multi-scale reaction system admits at most constant order of reaction intensities under a certain short-term timescale. For a multiscale process $X^N(t)$ associated with reaction system $(\S,\C,\Re,\K)$, let $\theta_0$ be the maximum order of reaction intensities. That is,
\begin{align}
\begin{split}\label{def:theta0}
    \theta_0&=\max\{ \theta_k : N^{\beta_k} \lambda_k(X^N(0)) =\Theta(N^{\theta_k})  \}=\max_k\{ \beta_k+y_k\cdot \alpha \},
    \end{split}
\end{align}
where $\alpha=(\alpha_1,\dots,\alpha_d).$}

\blue{
We consider time-scaled model $Z^{N,\gamma}_i(t):=Z^N(N^\gamma t)$. Then by \eqref{eq:multi rep} with a change of variable, $\Zs$ satisfies that
\begin{align}
Z^{N,\gamma}_i(t)=Z^N(N^\gamma t)&=Z^N(0)+\sum_k Y_k\left ( \int_0^{N^\gamma t} \lambda^N_k(Z^N(s))ds \right)\frac{(y'_{k,i}-y_{k,i})}{N^{\alpha_i}} \notag \\
&=Z^{N,\gamma}(0)+\sum_k Y_k\left ( \int_0^t \lambda^{N,\gamma}_k(Z^{N,\gamma}(s))ds \right)\frac{(y'_{k,i}-y_{k,i})}{N^{\alpha_i}} \label{eq:multi rep_time}
\end{align}}

\blue{ where 
\begin{align}
\begin{split}\label{eq:scaled intensitiy}
\lambda^{N,\gamma}_k(z)&=N^{\gamma+\beta_k} \prod_{i=1}^d N^{y_{k,i}\alpha_i}z_i\left( z_i-\frac{1}{N^{\alpha_i}} \right) \left (z_i-\frac{2}{N^{\alpha_i}} \right )\cdots \left (z_i-\frac{y_{k,i}-1}{N^{\alpha_i}}\right )\\
&=N^{\gamma+\beta_k+y_k\cdot \alpha } \prod_{i=1}^d z_i\left( z_i-\frac{1}{N^{\alpha_i}} \right) \left (z_i-\frac{2}{N^{\alpha_i}} \right )\cdots \left (z_i-\frac{y_{k,i}-1}{N^{\alpha_i}}\right )
\end{split}
\end{align}
 for $z \in \R^d_{\ge 0}$.  Note that since $Z^{N,\gamma}_i(0)=\Theta(1)$ for each $i$, each scaled intensity $\lambda^N_k(Z^{N,\gamma}(0))$ is  $\Theta(\gamma+\beta_k+y_k\cdot \alpha)$.
 For the set of intensities $\K=\{\lambda_k: y_k \to y'_k \in \Re\}$ of the original system $X$, we associate the set of scaled intensities $\lambda^{N,\gamma}_k$ in \eqref{eq:multi rep_time} with $Z^{N,\gamma}$ and denote by $\K^{N,\gamma}$.}

\subsection{Projection of Multiscale Reaction Systmes}\label{subsec:projection}
 \blue{We can reduce $(\S,\C,\Re,\K^N)$ to consider only the dynamics of a subset of $\S$ by using network projection, which broadly speaking consists of the removal of species from the network as described for the example \eqref{eq:example in intro} in Introduction, and the subsequent merging of complexes if needed.}

For a given system $(\S,\C,\Re,\K)$,  to formally define the network projection we introduce two projection functions for complexes and reactions in $(\S,\C,\Re)$.  Let $X^N$ be the stochastic process associated with $(\S,\C,\Re,\K)$.
We decompose the set of species as $\S=\S_L\cup \S_H$, where $\S_L=\{S_i \in \S : X^N_i(0)=\Theta(1)\}$ and $\S_H=\S\setminus \S_L$ correspond to species with low and high initial counts, respectively. We enumerate them as $\S_L=\{S_1,S_2,\dots,S_d\}$ and $\S_H=\{S_{d+1},S_{d+2},\dots,S_{d+r}\}$. Let $q_L : \Z^{d+r}\to \Z^d$ and $q_H : \Z^{d+r}\to \Z^r$ be projection functions such that for each $v=(v_1,\dots,v_d,v_{d+1},\dots,v_{d+r})^T\in \Z^{d+r}$,
\begin{align}\label{eq:project function}
q_L(v)=(v_1,v_2,\dots,v_d)^T\in \Z^d \quad \text{and} \quad q_H(v)=(v_{d+1},v_{d+2},\dots,v_{d+r})^T\in \Z^r.
\end{align}

 We demonstrate the usage of $q_L$ and $q_H$ with the network \eqref{eq:example in intro} for which we set $\S_L=\{A,C\}$ and $\S_H=\{B\}$. Since the associated vector for the complex $y=A+B$ in \eqref{eq:example in intro} is $y=(1,1,0)^T$, we have $q_L(y)=(1,0)$ and $q_H(y)=1$. Using a slight abuse of notation, we also denote $q_L(A+B)=A$ and $q_H(A+B)=B$. In the same way, for the reaction $A+B\to \emptyset$, $q_L$ defines the projected reaction $q_L(A+B)\to q_L(\emptyset)$, which is identical to $A \to \emptyset$.

\blue{Let  $(\S,\C,\Re,\K^N)$ be a given multiscale system. Then by using $q_L$ and $q_H$ we define the projected system $(\S_L,\C_L,\Re_L,\K_L)$. Reactions in $\Re_L$ are chosen pertaining to the scale of the reaction intensities in $\K^N$ because reactions of lower order intensities can be neglected. Let $\Zs(t)$ be the stochastic process associated with $(\S,\C,\Re,\K^N)$ and let $\theta_0$ be the maximum reaction intensity order \eqref{def:theta0}. Then we decompose 
\begin{align}\label{eq:decompose R}
    \Re=\Re_0\cup \Re^c_0 \quad \text{where \ $\Re_0=\{y_k\to y_k' : \lambda_k(Z^N(0))=\Theta(\theta_0)\}$},
\end{align}
and we define the set of projected reactions as
\begin{align}\label{eq:project network2}
    \Re_L=\{q_L(y_k)\to q_L(y'_k) : \text{$y_k\to y_k' \in \Re_0$ and $q_L(y_k)\neq q_L(y'_k)$}\}.
\end{align} The set of complexes of the projected network $\C_L$ is fully characterized with the complexes involved in the reactions in $\Re_L$.
}

\blue{To defined the reaction intensities of the scaled process $Z^{N,\gamma}(t)$ associated with the projected network, we first decompose the reaction intensities defined in \eqref{eq:scaled intensitiy} and then we fix the species in $S_H$ at their initial state. For a given $\K^N$, each mass-action intensity $\lambda^{N,\gamma}_k$ for a reaction $y_k\to y_k'$ is decomposed as
$\lambda^{N,\gamma}_k(z)=\kappa_k \lambda_{L,k}(q_L(z))\lambda^{N,\gamma}_{H,k}(q_H(z))$ for each $z\in\R^{d+r}_{\ge 0}$, where
\begin{align}
\begin{split}\label{eq:decomp of lam}
    &\lambda_{L,k}(z)=q_L(z)^{(q_L(y_k))}, \text{\ and \ } \\ &\lambda^{N,\gamma}_{H,k}(z)=N^{\gamma+\beta_k+y_k\cdot \alpha}\prod_{i=d+1}^{d+r} z_i\left (z_i-\frac{1}{N} \right)\cdots \left (z_i-\frac{y_{k,i}-1}{N} \right).
\end{split}
\end{align}
 Let $Z^{N,\gamma}(0)=z^0=(z^0_\ell,z^0_h)$ such that $z^0_\ell \in \mathbb{Z}^d_{\ge 0}$ and $z^0_h\in \R^r_{\ge 0}$. Then by fixing $q_H(\Zs(t))$ at $z^0_h$, we define the reaction intensities of the projected system as
\begin{align}\label{eq:rate constants for reduced}
\K_L=\left \{\bar \lambda_{u}(x)=\bar \kappa_u x^{(\bar y_u)} :  \bar y_u \to \bar y'_u \in \Re_L \right \},
\end{align}
where
\begin{align*}
    \bar \kappa_u=\sum_{\substack{y_k\to y'_k \in \Re \\ q_L(y_k)=\bar y_u, q_L(y'_k)=\bar y'_u}} \kappa_k s_k, \quad \text{and} \quad s_k= \dlim_{N\to \infty} \lambda^{N,\gamma}_{H,k}(z^0_h).
\end{align*}
 $\bar \kappa_u$ serves a reaction rate constant of the projected system.
 Note that each $\bar \lambda_u$ in $\K_L$ is scale-free.
 \begin{remark}
  The summation in the definition of $\bar \kappa_u$ is to consider the case that multiple reactions in $\Re$ are projected in to a single reaction $\bar y_u \to \bar y'_u$ in $\Re_L$. 
  \end{remark}
 \begin{remark}\label{rmk:alternative K^N}
 For each $x\in\Z^d_{\ge 0}$, by definition of $\lambda_{L,k}$ we have $x^{(\bar y_u)} = \lambda_{L,k}(x)$ for any $k$ such that $q_L(y_k)=\bar y_u$. Hence letting $k_u$ be such that $q_L(y_{k_u})=\bar y_u$, $\K_L$ can be represented differently as $\K_L=\left \{\bar \lambda_{u}(x)=\bar \kappa_u \lambda_{k_u}(x) :  \bar y_u \to \bar y'_u \in \Re_L \right \}$.
 \end{remark}
 \begin{remark}\label{rmk:sk}
  If $y_k \to y'_k \in \Re_0$, then $-\theta_0+\beta_k+y_k \cdot \alpha=0$ by the definition of $\theta_0$ in \eqref{def:theta0}. Hence for $y_k \to y'_k \in \Re_0$
 \begin{align*}
      s_k&=\lim_{N\to \infty}\lambda^{N,-\theta_0}_{H,k}(q_H(Z^{N,-\theta_0}(0))\\
      &=\lim_{N\to \infty}N^{-\theta_0+\beta_k+y_k\cdot \alpha}\prod_{i=1}^r z^0_{h,i}\left (z^0_{h,i}-\frac{1}{N}\right )\cdots \left (z^0_{h,i}-\frac{q_H(y_k)_i-1}{N}\right )   = {(z^0_h)}^{q_H(y_k)}.
 \end{align*}
 If $y_k\to y'_k \in \Re^c_0$, otherwise, then $s_k=0$ since $-\theta_0+\beta_k+y_k\cdot \alpha<0$.
 \end{remark}
 }

We demonstrate the projection of a reaction system with an example.
\begin{example}\textnormal{\blue{
Consider a stochastic process $X^N(t)$ associated with a reaction network $(\S,\C,\Re,\K)$ such that
\begin{align*}%\label{eq:example1}
 A+B\xrightleftharpoons[\kappa_2]{\kappa_1/N}  2B, \quad A+C \xrightarrow{\kappa_3} 2C, \quad 3C \xrightleftharpoons[\kappa_5 N]{\kappa_4/N^2} A, \quad B\xrightarrow{N\kappa_6} \emptyset.
\end{align*}
Suppose that $X^N_A(0)=X^N_B(0)=2$ and $X^N_C(0)=3N$. Then to reduce the scaled system $Z^N$, note that $\S=\S_L\cup \S_H$ with $S_L=\{A,B\}$ and $\S_H=\{C\}$.  To find $\Re_L$, note that $\theta_0=1$ and the reaction intensities of $A+C\to 2C, 3C\rightleftharpoons A$ and $B\to \emptyset$ are belonging to $\Re_0$. Therefore, by projecting those reactions with $q_L$, we obtain
\begin{align*}
    \Re_L=\{A \to \emptyset, \emptyset \to A, B\to \emptyset\},
\end{align*}
here note that two reactions $A+C\to 2C$ and $A\to 3C$ are merged into the same reaction $A\to \emptyset$.
Finally, by \eqref{eq:rate constants for reduced}
\begin{align*}
    \K_L=\{\bar \lambda_{A\to \emptyset}(x)=(s_3\kappa_3+s_5\kappa_5)x_A, \bar \lambda_{\emptyset\to A}(x)=s_4\kappa_4, \bar \lambda_{B\to \emptyset}(x)=s_6\kappa_6 x_B\}
\end{align*}
 can be defined for each $x=(x_A,x_B)\in \Z^2_{\ge 0}$ by freezing $\Zt_C(t)$ at $\Zt_C(0)=3$. 
 To compute $s_3$, note that $\beta_3=0$ and $y_3\cdot \alpha=1$. Then as shown in Remark \ref{rmk:sk}, for the initial condition $z^0=(1,2,3)$ of the scaled process $Z^{N,-\theta_0}$
\begin{align*}
    &s_3=q_H(z^0)^{q_H(y_3)}=3, \quad s_4=q_H(z^0)^{q_H(y_4)}=27\\
    &s_5=q_H(z^0)^{q_H(y_5)}=1, \quad s_6=q_H(z^0)^{q_H(y_6)}=1.
\end{align*}
%  \begin{align*}
%      s_3=\lim_{N\to \infty}\lambda^{N,-\theta_0}_{H,3}(q_H(Z^{N,-\theta_0}(0)))=\lim_{N\to \infty} N^{-1+0+1}Z^{N,-\theta_0}_C(0)=3
%  \end{align*}
%  In the same way, 
%  \begin{align*}
%      s_4&=\lim_{N\to \infty}\lambda^{N,-\theta_0}_{H,4}(q_H(Z^{N,-\theta_0}(0)))\\
%      &=\lim_{N\to \infty} N^{-1-2+3}Z^{N,-\theta_0}_C(0)\left(Z^{N,-\theta_0}_C(0)-\frac{1}{N}\right)\left(Z^{N,-\theta_0}_C(0)-\frac{2}{N}\right)=27\\
%      s_5&=\lim_{N\to \infty}\lambda^{N,-\theta_0}_{H,5}(q_H(Z^{N,-\theta_0}(0)))=\lim_{N\to \infty} N^{-1+1+0}=1\\
%      s_6&=\lim_{N\to \infty}\lambda^{N,-\theta_0}_{H,5}(q_H(Z^{N,-\theta_0}(0)))=\lim_{N\to \infty} N^{-1+1+0}=1.
%  \end{align*}
 Therefore $(\S_L,\C_L,\Re_L,\K_L)$ is described with 
\begin{align*}
A\xrightleftharpoons[27\kappa_4]{3\kappa_3+\kappa_5}  \emptyset \xleftarrow{\kappa_6} B
\end{align*}}}
\end{example}

\section{Main Results}\label{sec:main}
In this section, we introduce our main results. In \cite{ball2006asymptotic, KangKurtz2013}, it was shown that if the scaling exponents $\alpha_i$, $\beta_k$, and $\gamma$ in \eqref{eq:multi rep_time} satisfy certain balance conditions, then species of high abundance follow a system of differential equations with random coefficients, and the species of low abundance follow a piece-wise deterministic Markov process. \blue{In this paper, for a given $X^N(t)$ with scaling parameters $\alpha_k$, $\beta_k$,  we consider $=\Zt(t)$ under slow-timescale. Under this timescale, we show that the species in $\S_L$ approximately follow the projected system $(\S_L,\C_L,\Re_L,\K_L)$ as the scaling parameter tends to infinity. We further investigate the accuracy of this approximation, which has not been investigated in the previous work.}
 %These special properties may not be guaranteed if only the usual weak convergence for probability measures holds. In order to investigate these results, we assume that the projected system admits a stationary distribution with a fast-decaying tail. 

\subsection{Main Theorem}
\blue{For a scaled process $\Zt(t)$ associated with $(\S,\C,\Re,\K^N)$, let $Z(t)$ be the stochastic process associated with the projected system $(\S_L,\C_L,\Re_L,\K_{L,s})$ as defined in Section \ref{subsec:projection} such that $Z(0)=q_L(\Zs(0))$. We denote by $p^N(\cdot,t)$ and $p(\cdot,t)$ the probability density of $\Zt$ and $Z$, respectively. 
 Throughout this paper, we always assume that $\S=\S_L\cup S_H$ such that $\S_L=\{S_i\in \S\ | \ X^N_i(0)=\Theta(1)\}$ and $\S_H=\{ S_i\in \S \ | \ X^N_i(0)=\Theta(N)\}$.} We further assume that the stochastic system $Z(t)$ associated with the projected network $(\S_L,\C_L,\Re_L,\K_{L})$ is irreducible and admits a stationary distribution $\pi$ such that 
  \begin{align}\label{condition3:tail} 
  \sum_{x\in \Z^d_{\ge 0} }\sum_{u} \bar \lambda_u (x)^2 \pi(x) < \infty.
\end{align}
\blue{This condition is required to exclude irregular behavior of $Z$ and in turn $Z^{N,-\theta_0}$ such as explosion.}

\begin{theorem}\label{thm:main1}
For each $t$, $\Zt(t)$ converges to $Z$ in distribution as $N\to \infty$. Furthermore, there exists constants $c >0$ and $\nu \in (0,1)$ such that for any measurable set $A \subset \R^{d+r}_{\ge 0}$
\begin{align*}
    \sup_{t\in [0,T]}\left |p^N\left (A ,t \right )-p(A_L,t) \right | \le\frac{c \max\{1,T^2\}}{N^\nu} \quad \text{for any $T>0$},
\end{align*}
where $A_L=\{q_L(z)  \ | \ z\in A \}$.
\end{theorem}

\begin{remark}
Two lemmas are required to complete the proof of Theorem \ref{thm:main1}. We introduce the lemmas in Section \ref{sec:lemma}.
\end{remark}
\begin{proof}
%In this proof, we denote $\dfrac{t}{N^{\theta_0}}$ by $t_N$.
 Lemma \ref{lem:explosion} shows that for any $M>0$ there exists a compact set $S_M=S_{L,M}\times S_{H,M}$ satisfying (i) $S_{L,M}\subset \Z^d_{\ge 0}$, $S_{H,M}\subset \R^r_{\ge 0}$ and (ii) for any $t$ there exists $c_1>0$ such that
\begin{align}\label{eq:prob of outside}
p^N(S_M^c,t)  \le \frac{c_1 \max\{1,t^2\}}{M^2} \quad \text{and} \quad p(S^c_{L,M},t) \le \frac{c_1 \max\{1,t^2\}}{M^2}.
\end{align}
 By using this we split the set $A$ as $A = ( A\cap S_M) \cup ( A\cap S^c_M )$.
Thus we have
\begin{align*}
\begin{split}
&|p^N(A,t) - p(A_L,t)| \\
&=|p^N(A\cap S_M ,t) +p^N(A\cap S^c_M  ,t)- p(A_L\cap S_{L,M},t)-p(A_L\cap S_{L,M}^c,t)|\\ 
&\le |p^N\left ( A\cap S_{M},t)-p(A_L\cap S_{L,M},t\right )|+p^N(S_M^c,t) + p(S_{L,M}^c,t).
\end{split}
\end{align*}
 Lemma \ref{lem:for densities} shows that there exist positive constants $c_2$ and $c_3$ such that
 \begin{align}\label{eq:the target term}
    |p^N\left ( A\cap S_{M},t)-p(A_L\cap S_{L,M},t\right )| \le  \dfrac{c_2 M^{c_3}}{N}\max\{1,t^2\}
 \end{align}
Thus if we choose $M=N^\rho$ for some $\rho < \dfrac{1}{c_3}$, then by \eqref{eq:prob of outside} and \eqref{eq:the target term}, the result follows with $c=c_1+c_2$ and $\nu=\min\{2\rho,1-\rho c_3\}$. 
 \end{proof}

Figure \ref{fig1} shows a schematic procedure of the proof of Theorem \ref{thm:main1}.

\begin{figure}[h!]
\includegraphics[width=\textwidth]{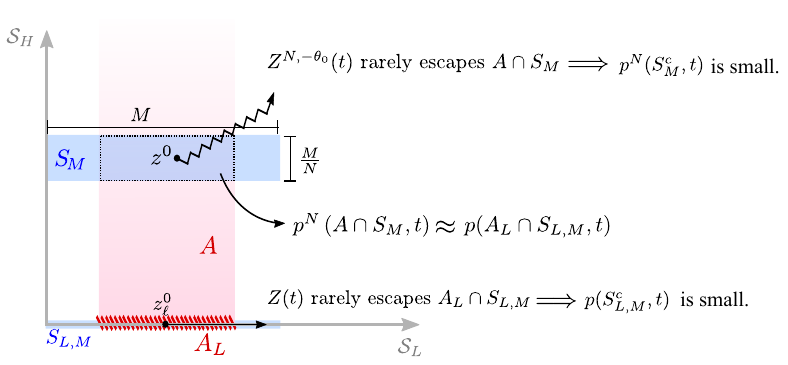}
\caption{A schematic description of the proof of the main theorem.}\label{fig1}
\end{figure}

\subsection{Lemmas}\label{sec:lemma}

In this section, we provide two main lemmas (Lemma \ref{lem:explosion} and \ref{lem:for densities}) used in the proof of Theorem \ref{thm:main1}. We further introduce additional lemmas required to proof the main lemmas. 
% We begin with useful notations for the proof of the lemmas. For measurable subsets $A_L\subset \Z^d_{\ge 0}$ and $A_H \subset \R^r_{\ge 0}$, we denote the probability densities by
% \begin{align*}
%      &p^N(A_L\times A_H,t)=P\left (q_L(Z^N(t)) \in A_L,\  q_H(Z^N(t)) \in A_H \right ). 
% \end{align*}
%  For a state $z$ of the stochastic process $Z^N$, we denote $z=(z_\ell,z_h)$ such that $q_L(z)=z_\ell \in \Z^d_{\ge 0} $ and $q_H(z)=z_h \in \R^r_{\ge 0}$.
 We use the usual $\infty$-norm and $1$-norm for vectors in $\Z^m$. That is, for $v\in \Z^m$ 
\begin{align*}
\Vert v \Vert_{\infty}=\max\{v_i:i=1,2,\dots,m\} \quad \text{and} \quad 
\Vert v \Vert_1=\sum_{i=1}^m |v_i|. 
\end{align*} 
 Note that the state space of $Z^{N,-\theta_0}$ is a subset of $\Z^{d}_{\ge 0}\times \R^r_{\ge 0}$. 
Note further that the state space of $(Z^N_{d+1},\dots,Z^N_{d+r})$ depends on the scaling parameter $N$, as $Z^{N,-\theta_0}_i(t)=\frac{X^N_i(t/N^{-\theta_0})}{N}$ for $i=d+1,\dots,d+r$.
We denote the state space of $Z^{N,-\theta_0}$ by $\mathbb{S}_\ell\times \mathbb{S}_h$ such that $\mathbb S_\ell \subseteq \Z^d_{\ge 0}$ and $\mathbb S_h \subseteq \R^r_{\ge 0}$, where $q_L(\Zt (t))\in \mathbb{S}_\ell$ and $q_H(\Zt (t)) \in \mathbb{S}_h$.

In the rest of this manuscript, for $z\in \R^{d+r}_{\ge 0}$ we denote $z_\ell =q_L(z)$ and $z_h=q_H(z)$. Then we define $S_M=S_{L,M}\times S_{H,M}$ such that
\begin{align}
\begin{split}\label{eq:compact sets}
    &S_{L,M}=\{z_\ell \in \mathbb{S}_\ell : \Vert z_\ell \Vert_{\infty} \le M\}, \quad \text{and} \\
    &S_{H,M}=\left \{ z_h \in \mathbb{S}_h : \left |\lambda^N_{H,k}(z_h) -s_k \right | \le \frac{M}{N}  \text{\ \ for any $y_k\to y_k' \in \Re_0$} \right \},
\end{split}
\end{align}
where $M=N^{\rho}$ for arbitrary $\rho\in (0,1)$.
As shown in Figure \ref{fig1}, one of the key ideas for the main theorem is to show that $Z^N$ stays in the compact set $S_M$ within a finite time interval $[0,T]$ with high probability. In the following lemmas, we show that the probability of $\Zt$ escaping $S_M$ is low if $N$ is sufficiently large.

\begin{lemma} \label{lem:two intensities}
 For each $y_k\to y'_k\in \Re$, let $\bar y_u\to \bar y'_u \in \Re_L$ such that $q_L(y_k)=\bar y_u$ and $q_L(y'_k)=\bar y'_u$. Then there exist $
 \nu_1, \nu_2, c>0$ such that for any $z\in S_M$
\begin{align*}
 &\text{(i)}\quad  \left | \lambda^{N,-\theta_0}_{k}(z)-\frac{s_k \kappa_k}{\bar \kappa_u}\bar \lambda_{u}(z_\ell)   \right |\le  \kappa_k \lambda_{L,k}(z_\ell)\frac{1}{N^{\nu_1}}  \hspace{0.7cm} \text{if $y_k\to y'_k\in \Re_0, $} \\
 &\text{(ii)}\quad   \lambda^{N,-\theta_0}_{k}(z) \le \frac{c }{N^{\nu_2}}  \hspace{4.8cm} \text{if $y_k\to y'_k\in \Re^c_0$, and}\\
 &\text{(iii)}\quad   \left(1-\frac{c}{N^{\nu_1}} \right ) \sum_u\bar \lambda_u(z_\ell) \le \sum_k \lambda^{N,-\theta_0}_k(z)\le \left(1+\frac{c}{N^{\nu_1}} \right ) \sum_u\bar \lambda_u(z_\ell)+\frac{c}{N^{\nu_2}}.
\end{align*}
for $N$ sufficiently large, where $s_k$ and $\bar \kappa_u$ are defined as \eqref{eq:rate constants for reduced}.
\end{lemma}

The following lemma shows that the number of transitions of $Z$ by time $t$ is on average less than $\max\{1,t^2\}$ assuming \eqref{condition3:tail}. 
% \begin{lemma}\label{lem:monotone}
% Let $X$ be a continuous time Markov process with the Kolmogorov forward equation \eqref{eq:master}.
% Let $\pi$ be a stationary distribution of $X$, and let $\mu$ be an initial distribution of $X$. Suppose there exists $c>0$ such that $\mu(x)\le c \pi(x)$ for any $x$. Then for each $t$ and for any $x$, we have
% $P(X(t)=x)\le c\pi(x)$.
% \end{lemma}

\begin{lemma}\label{lem:mean of the jump number}
 Let $J(t)$ be the number of jumps of $Z$ by time $t$. If \eqref{condition3:tail} holds, then there exists a constant $c$ such that
\begin{align*}
E(J(t)^2)\le c \max\{1,t^2\} \quad \text{for any $t$}.
\end{align*}
\end{lemma}

\begin{remark}
This guarantees that $Z$ is \emph{non-explosive} 
meaning that if we let $T_n$ be the $n$ the transition time of $Z$, then $\dlim_{n\to \infty} T_n=\infty$ almost surely. Then $Z$ does not transition infinitely many times on any finite time interval $[0,T]$, which means $Z$ is uniquely well-defined process satisfying \eqref{eq:kurtz rep} with the intensity functions $\bar \lambda_u$.
\end{remark}

Now we investigate how many transitions are required  for $Z^{N,-\theta_0}(t)$  to escape the set $S_M$.
In the following lemmas, we use the floor function $\lfloor x \rfloor$ that is the greatest integer less than $x$.
\begin{lemma}\label{lem:number of jumps needed}
Let $\tau^N_r$ be the first time for $\lfloor r \rfloor$ th transition of $Z^{N,-\theta_0}$. Then there exists a constant $c>0$ such that 
$p^N(S^c_M,t)\le P(\tau^N_{cM}<t)$ for $M=N^\rho$ with arbitrary $\rho\in (0,1)$  as long as $N$ is sufficiently large.
\end{lemma}

\begin{remark}\label{rmk:tau for Z}
Let $\tau_r$ be the first time for $\lfloor r \rfloor$ th transition of $Z$. Then as shown in the proof of Lemma \ref{lem:number of jumps needed}, $\{z_\ell : |z_\ell - z^0_\ell| \le c'''M\} \subseteq S_{L,M}$ for some $c'''>0$. Therefore $p(S_{L,M},t)\le P(\tau_{c'''M} < t)$.
\end{remark}
Now, Lemma \ref{lem:two intensities} and Lemma \ref{lem:number of jumps needed} are combined to show \eqref{eq:prob of outside}. Before we introduce Lemma \ref{lem:explosion}, we remind that $Z(t)$ is irreducible. Therefore every state of $Z$ is non-absorbing meaning that $\sum_u \bar \lambda_u(z_\ell) > 0$ for any state $z$. Since $\bar \lambda_u$ is a non-zero polynomial as defined in \eqref{eq:rate constants for reduced}, it follows that
\begin{align}\label{eq:min bar lambda}
    \min_{z_\ell} \sum_u \bar \lambda_u(z_\ell) > 0 \text{, and so }  \min_{z} \sum_k \lambda^{N,-\theta_0}_k(z) > 0. 
\end{align}

\begin{lemma}\label{lem:explosion}
For any $t$ there exists $c>0$ and $\nu_0 > 0$ that
\begin{align*}
p^N(S_M^c,t)  \le \frac{c_1 \max\{1,t^2\}}{N^{\nu_0}} \quad \text{and} \quad p(S^c_{L,M},t) \le \frac{c_1 \max\{1,t^2\}}{N^{\nu_0}}.
\end{align*}
% Suppose the same conditions in Lemma \ref{lem:two intensities} and Lemma \ref{lem:mean of the jump number}.
% Let $\tau^N_n$ and $\tau_n$ be the time of $n$-th jump of $Z^{N}$ and $Z$, respectively. Let $t$ be a fixed time. Then for $N$ sufficiently large, there is a constant $c>0$ such that
% \begin{align} \label{eq:prob of tau}
% P^N(\tau^N_M < t) \le P(\tau_M < 2t) \le  \left( \frac{\left ( 1+\frac{c M}{N}\right )}{\left ( 1-\frac{c M}{N}\right )} \right )^M\frac{c \max\{1,t^2\}}{M^2}.
% \end{align}
% In particular, if $M=f(N)$ such that $\lim_{N\to \infty} \frac{f(N)^2}{N}=0$, then 
% \begin{align}\label{eq:prob of tau2}
% P(\tau^N_M < t) \le P(\tau_M < 2t)  \le \frac{c' \max\{1,t^2\}}{M^2}.
% \end{align}
% for some $c'>0$.
\end{lemma}

\begin{proof}%[\textbf{Proof of Lemma \ref{lem:explosion}}]
Let the two stopping times $\tau^N_M$ and $\tau_M$ be defined as Lemma \ref{lem:number of jumps needed} and Remark \ref{rmk:tau for Z}. Then $p^N(S^c_M,t) \le P(\tau^N_{cM} < t)$ and $p(S^c_{L,M},t) \le P(\tau_{cM} < t)$ for some $c>0$. Hence we show the bounds for $P^N(\tau^N_{cM} < t)$ and $P(\tau_{cM} < t)$.

Let $Path_{M'}$ be a collection of all possible $\lfloor cM\rfloor=M'$ consecutive reactions in $\Re$ for $Z^{N}$ started at $Z^{N}(0)=(z^0_\ell,z^0_h)$. Each element $\eta \in Path_{M'}$ is an ordered set of $M'$ reactions in $\Re$. That is, 
\[\eta=\left \{y(\eta,1)\to y'(\eta,1),\dots, y(\eta,M')\to y'(\eta,M')\right\},\] where $y(\eta,i)\to y'(\eta,i) \in \Re$ for each $i$. We define $w(\eta,j)$ be a state after $j$ consecutive jumps in $\eta$ from $Z^{N,-\theta_0}(0)=(x_\ell,x_h)$. That is, 
\begin{align}
w(\eta,j)&=Z^N(0)+\sum_{i=1}^j (y'(\eta,i)-y(\eta,i))=(w(\eta,j)_\ell,w(\eta,j)_h)
\end{align}
for $j=1,2,\dots,M'$, where 
\begin{align*}
&w(\eta,j)_\ell=\left (z^0_\ell+ \sum_{i=1}^j q_L(y'(\eta,i)-y(\eta,i)) \right)\text{ and }\\
&w(\eta,j)_h= \left (z^0_h+\sum_{i=1}^j q_H(y'(\eta,i)-y(\eta,i))\right ).
\end{align*}
Note that $w(\eta,j) \in S_M$ since we choose $c$ as $Z^{N,-\theta_0}$ needs at least $M'$ transitions to escape $S_M$.
Let $Path_{M',0}=\{ \eta : y(\eta,i)\to y'(\eta,i) \in \Re_0 \text{ for each $i$} \}$.  Then for $\eta \in Path_{M',0}$ we denote by $A_\eta$ represent the event of $M'$ consecutive jumps for $Z^{N,-\theta_0}$ along the ordered reactions in $\eta \in Path_{M'}$. Let also $\bar A_\eta$ the event of $M'$ consecutive jumps for $Z$ along the ordered reactions $q_L(y(\eta,1))\to q_L(y'(\eta,1)),\dots,q_L(y(\eta,M'))\to q_L(y'(\eta,M'))$.

We now show two key steps. First, \blue{conditioning} on the event $A_\eta$, the stopping time $\tau^N_{cM}$ is sum of exponential distributions $T^N_i$ with rate $\lambda^N(w(\eta,j))$ \cite{NorrisMC97}, where $\lambda^{N,-\theta_0}(z)=\sum_{k} \lambda^{N,-\theta_0}_k(z)$. More precisely,
\begin{align}\label{eq:tau is sum of expo}
P(\tau^N_{cM} <t | A_\eta )=P\left (\sum_{i=1}^{M'} T^N_i < t\right),
\end{align}
where $T^N_i$'s are independent exponential distributions with the rate $\lambda^{N,-\theta_0}(w(\eta,i))$. Lemma \ref{lem:two intensities} (iii) implies that $\lambda^N(z) \le 2 \bar \lambda(z_\ell)$ for $z \in S_{M}$, where $\bar \lambda(z_\ell)=\sum_u \bar \lambda(z_\ell)$. Then this combined with \eqref{eq:tau is sum of expo} implies that 
\begin{align}\label{eq:jump time distribution ineq}
P(\tau^N_{cM} <t | A_\eta ) \le P \left (\sum_{i=1}^{M'} T_i < t/2 \right )= P\left (\tau_{cM} < t/2 | \bar A_\eta \right ),
\end{align}
where $T_i$'s are independent exponential distributions with the rate $\sum_u\bar \lambda_u(w(\eta,i)_\ell)$.

Second, note that for each $\eta \in Path_{M'}$, the event $A_\eta$ occurs if and only if the reaction $y(\eta,i)\to y'(\eta,i)$ fires at each state $w(\eta,i)$ among all reactions in $\Re$. This implies that
\begin{align*}%\label{eq: prob of paths}
P(A_\eta)= \prod_{i=1}^{M'} \frac{\lambda^{N,-\theta_0}_{\eta,i}(w(\eta,i))}{\lambda^{N,-\theta_0}(w(\eta,i))},
\end{align*} 
where $\lambda^N_{\eta,i}$ is the intensity function of the reaction $y(\eta,i)\to y'(\eta,i)$. 
%As shown in Lemma \ref{lem:two intensities}, we split $\lambda^N_{\eta,i}(w(\eta,i))$ into two parts as
%\begin{align*}
%\lambda^N_{\eta,i}(w(\eta,i))=\kappa_{\eta,i}\lambda^N_{L,\eta,i}(w(\eta,i)_\ell)\lambda^N_{H,\eta,i}(w(\eta,i)_h).
%\end{align*}  
Lemma \ref{lem:two intensities} (i) and the fact that $\lambda_{L,k}(z_\ell)= \frac{1}{\bar \kappa_u}\bar \lambda_u(z_\ell)$  if $q_L(y_k)=\bar y_u$ imply that there exists $c>0$ such that for each $\eta \in Path_{M',0}$
\begin{align}\label{eq:two total intensities}
\lambda^{N,-\theta_0}_{\eta,i}(w(\eta,i))\le \left (\frac{s_k\kappa_k}{\bar \kappa_u}+\frac{\kappa_k}{\bar \kappa_u N^{\nu_1}}\right) \bar \lambda_{\eta,i}(w(\eta,i)_\ell) \le \left(1+\frac{c}{N^{\nu_1}} \right ) \bar \lambda_{\eta,i}(w(\eta,i)_\ell)
\end{align}
for some $c>0$, where $\bar \lambda_{\eta,i}$ denotes the intensity of $q_L(y(\eta,i))\to q_L(y'(\eta,i))$. Then
Lemma \ref{lem:two intensities} (iii) and \eqref{eq:two total intensities} further imply that there exists $c>0$ such that for each $\eta \in Path_{M',0}$
\begin{align}
\begin{split}\label{eq:bound of tau}
P(A_\eta)\le \left( \frac{\left ( 1+\frac{c}{N^{\nu_1}}\right )}{\left ( 1-\frac{c}{N^{\nu_1}}\right )} \right )^{cM}\prod_{i=1}^{M'} \frac{\bar \lambda_{\eta,i}(w(\eta,i)_\ell)}{\bar \lambda(w(\eta,i)_\ell)}=\left( \frac{\left ( 1+\frac{c}{N^{\nu_1}}\right )}{\left ( 1-\frac{c}{N^{\nu_1}}\right )} \right )^{cM} P(\bar A_\eta) ,
\end{split}
\end{align}
For $\eta \in Path^c_{M',0}$, a reaction $y(\eta,i)\to y'(\eta,i)\in \Re^c_0$ for some $i$. 
Then 
\begin{align*}
    P\left (\bigcup_{\eta\in Path^c_{M',0}} A_\eta \right )&\le P(\text{a reaction in $R^c_0$ fires at a state lying in $S_M$})\\
    &\le \max_{z\in S_M}\max_{y_k\to y'_k \in R^c_0} \frac{\lambda^{N,-\theta_0}_k(z)}{\lambda^{N,-\theta_0}(z)}.
\end{align*}
Hence by Lemma \ref{lem:two intensities} (ii) and \eqref{eq:min bar lambda}
\begin{align}\label{eq:bound of tau2}
P\left (\bigcup_{\eta\in Path^c_{M',0}} A_\eta \right ) \le \frac{c}{N^{\nu_2}}
\end{align}
for some $c>0$.
%that $\lambda^N_{\eta,i}(w(\eta,i)) \le \left(1+c\frac{M}{N}\right ) \bar \lambda_{\eta,i}(w(\eta,i)_\ell)$,
%where $\bar \lambda_{\eta,i}$ is the intensity of reactions $y_u\to y'_u$ such that $q_L(y_u)=q_L(y(\eta,i))$ and $q_L(y'_u)=q_L(y(\eta,i))$ in the projected system $(\S_L,\C_L,\Re_L,\K_L)$.
% Moreover

Consequently,  by \eqref{eq:jump time distribution ineq}, \eqref{eq:bound of tau} and \eqref{eq:bound of tau2} there exists $c>0$ such that
\begin{align}
    P(\tau^N_{cM} <t) &= P\left (\tau^N_{cM} <t, \bigcup_{\eta\in Path_{M',0}} A_\eta \right ) +  P\left (\tau^N_{cM} <t, \bigcup_{\eta\in Path^c_{M',0}} A_\eta \right )\\
    &\le \sum_{\eta \in Path_{M',0}}P(\tau^N_{cM}<t | A_\eta)P(A_\eta)+\frac{c}{N^{\nu_2}}\notag \\
&\le \left( \frac{\left ( 1+\frac{c}{N^{\nu_1}}\right )}{\left ( 1-\frac{c}{N^{\nu_1}}\right )} \right )^{cM}\sum_{\eta \in Path_{M',0}}P( \tau_{cM} < t/2 | \bar A_\eta) P(\bar A_\eta) + \frac{c}{N^{\nu_2}}	\notag \\
&\le \left( \frac{\left ( 1+\frac{c}{N^{\nu_1}}\right )}{\left ( 1-\frac{c}{N^{\nu_1}}\right )} \right )^{rM} P(\tau_{cM} < t/2) + \frac{c}{N^{\nu_2}}. \label{eq:relation two taus}
\end{align}

Finally, we let $J(t)$ be the number of jumps of $Z$ by time $t$ as we define in Lemma \ref{lem:explosion}.
Then, applying the Chebyshev's inequality and the result of Lemma \ref{lem:explosion}, we have \begin{align}\label{eq:using the jump bound}
P( \tau_{cM} < t/2)\le P(J(t/2) > cM) \le \frac{E(J(t/2)^2)}{(cM)^2} \le \frac{c'\max\{1,t^2\}}{M^2},
\end{align}
for some $c'>0$.
Recall that $M=N^\rho$. Then we have
\begin{align}\label{eq:using the bound of taus}
\lim_{N\to \infty}\left( \frac{\left ( 1+\frac{c}{N^{\nu_1}}\right )}{\left ( 1-\frac{c}{N^{\nu_1}}\right )} \right )^{rM} = 1,
\end{align}
because $\nu_1=1-\rho$ as shown in the proof of Lemma \ref{lem:two intensities}. We further recall that $\nu_2$ in Lemma \ref{lem:two intensities} tends to $1$, as $\rho \to 0$. Thus we choose $\rho$ sufficiently small. Then by \eqref{eq:relation two taus}--\eqref{eq:using the bound of taus}, the desired bounds follow with some $\nu_0$.
\end{proof}

%Note that
%$\tau_M < t/2$ if and only if $J(t/2) > c''M$ for some $c''>0$ since $\max_{\bar y_u \to \bar y'_u\in \Re_L} \{\Vert \bar y'_u-\bar y_u\Vert_\infty\} < \infty$ so that $Z$ needs at least $c''N$ times of jumps to escape $S_{L,M}$.

With Lemma \ref{lem:explosion}, we can conclude that both $Z^{N,-\theta_0}$ and $Z$ likely stay in $S_M$ and $S_{L,M}$, respectively within $[0,t]$ as long as $N$ is sufficiently large. \blue{As the confined set $S_M$ and $S_{L,M}$ have finitely many states, by using this advantage we can compare the probability densities of the two processes confined onto $S_M$ and $S_{L,M}$, respectively. } To compare the two probability densities, we use the following multi dimensional Gronwall's inequality \cite{chandra1976linear}.

\begin{lemma}\label{lem:Gronwall}
Suppose for any vector $u_0 \in \R^n_{\ge 0}$ with $\Vert u_0 \Vert_1 =1$,  a system of differential equation 
\begin{align}\label{eq:system ode}
\begin{split}
\begin{cases}
\frac{d}{dt}u(t)=Au,\\
u(0)=u_0,
\end{cases}
\end{split}
\end{align}
admits a unique solution $u(t) \in \R^n_{\ge 0}$ such that $u_i(t)\ge 0$ for each $i$ and $\Vert u(t) \Vert_1=1$ for any $t$.  Suppose $v$ satisfies
$\frac{d}{dt}v(t)\le Av(t)+b$ for some $b \in R^n$. 
Then for each $t$, \begin{align*}
v(t) \le tb + t^2\bar Ab,
\end{align*}
where $\bar A$ is an $n\times n$ matrix such that $\bar A_{ij}=\max_{k,m}A_{km}$.
\end{lemma}

\blue{As two probability densities $p^N$ and $p$ solve two similar systems of ordinary differential equation, respectively, Lemma \ref{lem:Gronwall} helps finding the distance between the two densities.}

\begin{lemma}\label{lem:for densities}
For any $t$ and for any $A\subset \R^{d+r}_{\ge 0}$, there exist $c>0$, $c'>0$ and $\nu'_0>0$ such that
\begin{align*}
 |p^N\left ( A\cap S_{M},t)-p(A_L\cap S_{L,M},t\right )| \le  \dfrac{c  M^{c'}\max\{1,t^2\}}{N^{\nu'_0}},
\end{align*}
where $A_L=\{q_L(z) : z\in A\}$.
\end{lemma}
\begin{proof}

The probability density $p(z_\ell,t)$ of $Z(t)$ satisfies the Kolomogorov forward equation (chemical master equation) \eqref{eq:master}
\begin{align}\label{eq:master reduced}
    \frac{d}{dt}p(z_\ell,t)=\sum_{u}\bar \lambda_{u}(z_\ell-\bar y'_u +\bar y_u)p(z_\ell-\bar y'_u + \bar y_u,t)-\sum_{u}\bar \lambda_{u}(z_\ell)p(z_\ell,t). 
\end{align}
By considering $p(t)=\{p(z,t)\}_{z_\ell\in S_{L,M}}$ as a column vector, \eqref{eq:master reduced} is equivalent to
\begin{align}\label{eq:master for p}
    \frac{d}{dt}p(t)=\mathcal L^M p(t) + b_{in}^M-b_{out}^M,
\end{align}
where $L^M$ is an $|S_{L,M}| \times |S_{L,M}|$ matrix, and $b^M_{out}$ and $b^M_{in}$ are column vectors.
The $ij$ entry $L^M(i,j)$ with $i \neq j$ is the transition rate from the $j$ th state to the $i$ th state in $S_{L,M}$, and $L^M(i,i)=-\sum_{j} L^M(i,j)$. The the vector $b^M_{in}$ represents the in-flow from $S_{L,M}^c$ to $S_{L,M}$, and hence it is defined as for the i the state $z_\ell \in S_{L,M}$, the i th entry is
\begin{align}\label{eq:about out flow}
 b^M_{in}(i)=\sum_{\substack{\bar y_u\to \bar y_u' \\  z_\ell - \bar y_u'+\bar y_u \in S^c_{L,M} }}\bar \lambda_{u}(z_\ell-\bar y'_u +\bar y_u)p(z_\ell-\bar y'_u + \bar y_u,t)
\end{align}
For any $z_\ell \in S_{L,M}$, the state $z_\ell-\bar y'_u +\bar y_u$ belongs to $S_{L,rM}$ for some $r>1$. Hence $\bar \lambda_{u}(z_\ell-\bar y'_u +\bar y_u)\le c M^{\max_k \Vert \bar y_u \Vert_1}$ because  $\bar \lambda_u$ is a polynomial of degree $\Vert \bar y_u \Vert_1$. Moreover, by applying Lemma \ref{lem:number of jumps needed} and Lemma \ref{lem:explosion} for $rM$ instead of $M$, we have $p(z_\ell-\bar y'_u +\bar y_u,t) \le c\frac{\max\{1,t^2\}}{N^{\nu_3}}$ for some $\nu_0\in (0,1)$. Therefore with a sufficiently small $\rho$, each entry $b^M_{in}(i)$ is less than $\frac{c}{N^{\nu_3}}$ for some $c>0$ and $\nu_3\in (0,1)$. In the same way, we can show also that there exist  $c>0$ and $\nu_4\in (0,1)$ such that $b^M_{out}(i)\le\frac{c}{N^{\nu_4}}$ for each $i$. Hence we have the following componentwise inequality from \eqref{eq:master for p}:
\begin{align}\label{eq: ineq for p}
     \mathcal L^M p(t) - b^N_1 \le \frac{d}{dt}p(t)\le \mathcal L^M p(t) + b^N_1,
\end{align}
where $b^N_1$ is a column vector with each entry $\frac{c}{N^{\nu_3}}$.

Now we turn to the Kolomogorov forward equation for $p^N(z,t)$. We first recall that as defined in \eqref{eq:multi rep_time}, the reaction vector of $Z^{N,-\theta_0}$ is scaled so that the transition for each entry is $\frac{y'_{k,i}-y_{k,i}}{N^{\alpha_i}}$. Thus we denote this scaled reaction vector by $y'^N_k-y^N_k$. 
Then for each $z\in S_M$, the distribution $p^N(z,t)$ satisfies
\begin{align}\label{eq:masterN1}
\frac{d}{dt}p^N(z,t)=\sum_{k}\lambda^{N,-\theta_0}_k(z-{y'}^N_k-y^N_k)p^N(z-y'^N_k+y^N_k,t)-\sum_{k}\lambda^{N,-\theta_0}_k(z)p^N(z,t).
\end{align}
 Then we show that the two differential equations \eqref{eq:masterN1} and \eqref{eq:master reduced} are similar with Lemma \ref{lem:two intensities}. Note that as we discussed above, some state $z-y_k'+y_k$ in \eqref{eq:masterN1} may be outside $S_M$. However, such states can be encompassed by $S_{rM}$ for some $r>1$, and Lemma \ref{lem:two intensities} still holds for $S_{rM}$ with different constant $c$. Hence we can apply Lemma \ref{lem:two intensities} for each $z-y_k'+y_k \in S^c_M$.

%The right-hand side in \eqref{eq:masterN1} is equivalent to
% \begin{align*}
%   &\sum_{k}\kappa_k\lambda_{L,k}(z_\ell-q_L({y'}^N_k-y^N_k))\lambda^N_{H,k}(z_h-q_H({y'}^N_k-y^N_k))p^N (z-y'^N_k+y^N_k,t)\\
% &-\sum_{k}\kappa_k\lambda_{L,k}(z_\ell)\lambda^N_{H,k}(z_h)p^N(z_\ell,z_h,t)
% \end{align*}

Let $p^N_L(z_\ell,t)=\displaystyle \sum_{\substack{z\in A\cap S_M\\q_L(z)=z_\ell}} p^N(z,t)$, which represents the probability density of the projected process $q_L(Z^N(t))$.
For each $z\in S_{rM}$, if $y_k\to y_k'\in \Re_0$, then Lemma \ref{lem:two intensities} (i) implies that 
$\lambda^{N,-\theta_0}_k(z) \le \frac{s_k\kappa_k}{\bar \kappa_u}\bar \lambda_u(z_\ell) + \kappa_k \lambda_{L,k}(z_\ell)\frac{1}{N^{\nu_1}}$. Since $\lambda_{L,k}(z_\ell)\le (cM)^{\max_k \Vert y_k  \Vert_\infty}$ for any $z_\ell \in S_{L,rM}$, there exists $c>0$ such that for any $z_\ell\in S_{L,rM}$,
\begin{align}   
\begin{split} \label{eq:usage of lemma1}
&\sum_{\substack{z\in A\cap S_M\\q_L(z)=z_\ell}}\sum_{y_k\to y'_k\in \Re_0}\lambda^N_k(z-{y'}^N_k-y^N_k)p^N(z-{y'}^N_k-y^N_k,t) \\
&\le \sum_{u} \bar \lambda_u(z_\ell-\bar y'_u+\bar y_u)p^N_L(z_\ell-\bar y'_u+\bar y_u,t) + c \frac{M^{\max_k \Vert y_k \Vert_\infty}}{N^{\nu_1}}.
\end{split}
\end{align}
Lemma \ref{lem:two intensities} (ii) further implies that 
\begin{align}\label{eq:usage of lemma2}
    \sum_{\substack{z\in A\cap S_M\\q_L(z)=z_\ell}}\sum_{y_k\to y'_k\in \Re^c_0}\lambda^{N,-\theta_0}_k(z)p^N(z,t) \le \frac{cM^{\max_k \Vert y_k \Vert_\infty}}{N^{\nu_2}}
\end{align}
Hence by using \eqref{eq:usage of lemma1} and \eqref{eq:usage of lemma2}, we take $\sum_{\substack{z\in A\cap S_M\\q_L(z)=z_\ell}}$ in \eqref{eq:masterN1} to show that for sufficiently small $\rho$ there exists $c>0$ and $\nu_4 \in (0,1)$ such that for any $z_\ell \in S_{L,M}$ 
\begin{align*}
    &\frac{d}{dt}p^N_L(z_\ell,t)\\
    &=\hspace{-0.3cm}\sum_{\substack{z\in A\cap S_M\\q_L(z)=z_\ell}}\hspace{-0.2cm}\left(\sum_{k}\lambda^{N,-\theta_0}_k(z-{y'}^N_k-y^N)p^N(z-y'^N_k+y^N_k,t)-\sum_{k}\lambda^{N,-\theta_0}_k(z)p^N(z,t) \right) \\
    &\le \sum_{u}\bar \lambda_{u}(z_\ell-\bar y_u +\bar y_u)p^N(z_\ell-\bar y'_u + \bar y_u,t)-\sum_{u}\bar \lambda_{u}(z_\ell)p^N(z_\ell,t) + \frac{c}{N^{\nu_4}}.
\end{align*}
Hence as we derived \eqref{eq: ineq for p}, we have the following componentwise inequaility for the column vector $p^N_L(t)=\{p^N_L(z_\ell,t)\}_{z_\ell \in S_{L,M}}$: for some $c>0$ and $\nu_5 \in (0,1)$
\begin{align*}
     \mathcal L^M p^N_L(t) -b^N_2 \le \frac{d}{dt}p^N_L(t)\le \mathcal L^M p^N_L(t) + b^N_2,
\end{align*}
where $b^N_2$ is a column vetor with each entry $\frac{c}{N^{\nu_4}}$.
By combining these inequalities with \eqref{eq: ineq for p}, we have
\begin{align}\label{eq: ineq for pN-p}
   \mathcal L^M (p^N_L(t)-p(t))  - b^N \le    \frac{d}{dt}(p^N_L(t)-p(t)) \le  \mathcal L^M (p^N_L(t)-p(t)) + b^N,
\end{align}
with a vector $b^N$ such that each entry $b^N(i)$ is $\frac{c}{N^{\nu_5}}$ for some $c>0$ and $\nu_5\in (0,1)$.

Then we complete this proof applying Lemma \ref{lem:Gronwall} for \eqref{eq: ineq for pN-p}. 
  Note first that a system of differential equation 
  \begin{align*}
 \begin{cases}
\frac{d}{dt}u(t)=\mathcal L^Mu,\\
u(0)=u_0,
\end{cases}
  \end{align*}
  admits a unique solution $u(t)$ when $\Vert u_0 \Vert_1=1$ because we can regard $\mathcal L^M$ as the transition rate matrix of a continuous time Markov chain defined on $\S_{L,M}$. Hence applying Lemma \ref{lem:Gronwall}  for \eqref{eq: ineq for pN-p}  with $v(t)=p^N_L(t)-p(t)$ and $v(t)=p(t)-p^N_L(t)$ respectively,
 we have
 \begin{align*}
     |p^N_L(t)-p(t)|\le t b^N+t^2 |S_{L,M}|\mathcal L^M_{max} b^N,
 \end{align*}
 where $\mathcal L^M_{max}$ is the maximum entry of $\mathcal L^M$. Each entry of $\mathcal L^M$ is either the reaction intensity $\bar \lambda_u(z_\ell)$ or  finite sum of $\bar \lambda_u(z_\ell)$ at some $z_\ell \in S_{L,M}$. Hence $\mathcal L^M_{max}$ can be bound by $c M^{\max_u \Vert \bar y_u \Vert_1}$ for some $c>0$ because $\max_{u}\max_{z_\ell\in S_{L,M}}\bar \lambda_u(z_\ell) \le c M^{\max_u \Vert \bar y_u \Vert_1}$ for some $c>0$. Furthermore note that $|S_{L,M}| \le M^d$ and each entry of $b^N$ is $\frac{c}{N^{\nu_4}}$. Hence for each $z_\ell \in S_{L,M}$, we have
 \begin{align*}
     |p^N(A\cap S_M,t)-p(A_L\cap S_{L,M},t)|\le \sum_{z_\ell \in S_{L,M}}|p^N_L(z_\ell,t)-p(z_\ell,t)|\le\dfrac{c  M^{c'}\max\{1,t^2\}}{N^{\nu'_0}}
 \end{align*}
 with sufficiently small $\rho$ and some $\nu'_0\in (0,1)$.
\end{proof}

\section{Theorem \ref{thm:main1} with general kinetics}\label{sec:generalization}
Theorem \ref{thm:main1} can hold for a reaction system under general kinetics as long as the scaled reaction intensities \eqref{eq:scaled intensitiy} satisfy the following conditions.
\begin{enumerate}[label=(CD\arabic*)]
\item\label{condi1} 
The scaled reaction intensity for $Z^{N,\gamma}(t)$ is also decomposed as following: for $z$ such that $q_L(z)\in \Z^d_{\ge 0}$ and $Nq_H(z)\in \Z^{r}_{\ge 0}$
\begin{align*}
    \lambda^{N,\gamma}_k(z)=\kappa_k \lambda_{L,k}(q_L(z))\lambda^{N,\gamma}_{H,k}(q_H(z)).
\end{align*}

% \item\label{condi1-1} 
% Stoichiometric compatibility: $\lambda_k(x) =0$ if $x < y_k$ for each $k$, where the inequality is component-wise. Furthermore, $\min_{x\ge y_k}\lambda_k(x)>0$. 

 \item\label{condi2}  
$\lambda_{L,k}$ grows polynomially: for any $k$, there exist positive constants $c_1$ such that
\begin{align*}
\lambda_{L,k}(z_\ell)\le c_1 \Vert z_\ell \Vert_\infty^{c_2} \quad \text{for any $z_\ell \in \Z^d_{\ge 0}$}.
\end{align*} 

 \item\label{condi3}   The limit $\lim_{N\to \infty}\lambda^{N,\gamma}_{H,k}(z_h)$  exists for each $z_h\in \Re^r_{\ge 0}$ and we denote this limit by $\bar \lambda_{H,k}(z_h)$. Furthermore if $|z-z^0|\le \frac{M}{N}$, then 
  \begin{align*}
      |\bar \lambda_{H,k}(z_h)-s_k| \le \frac{M}{N} \quad \text{and} \quad |\lambda^{N,\gamma}_{H,k}(z_h)-\bar \lambda_{H,k}(z_h)| \le \frac{1}{N},
  \end{align*}
 where $s_k=\lim_{N\to \infty}\lambda^{N,\gamma}_{H,k}(z^0_h)$ with an initial condition $z^0_h \in \Re^{r}_{\ge 0}$,.
\end{enumerate}

\begin{remark}
Conditions \ref{condi1}--\ref{condi3} hold under the stochastic mass action kinetics \eqref{eq:mass}. 
\end{remark}
\begin{remark}
Polynomials, Michaelis–Menten kinetics, hill functions and logarithms satisfy \ref{condi2} and \ref{condi3}.
 \end{remark}

Note that Lemma \ref{lem:number of jumps needed} can be proved by using conditions \ref{condi1}--\ref{condi3} without properties of mass action kinetics. The other lemmas still hold without further modifications in the proofs. Thus for general kinetics satisfying conditions \ref{condi1}--\ref{condi3}, Theorem \eqref{thm:main1} holds.

% \begin{remark}
% In \cite{AndProdForm}, it was shown that the stochastic system associated with so-called complex balanced reaction networks, which arise in many biology systems, admits a product form of Poissons stationary distributions satisfying \ref{condition3:tail}.  
% \end{remark}

% \begin{remark}
% Under the stochastic mass action kinetics, \ref{condition1-1,Z=1} holds and $c_L$.
% \end{remark}

% \begin{remark}
% Under the stochastic mass action kinetics, \ref{condition2:evenly likely} implies that $\gamma+\beta_k=-\Vert q_H(y_k)\Vert _1$. A more general condition for $\gamma$ and $\beta_k$ called \textit{species balance condition} is introduced in \cite{KangKurtz2013}. 
% \end{remark}

\section{Examples}\label{sec:examples}

We apply Theorem \ref{thm:main1} for several multiscale biological systems. In the follow examples, the probability density of low order species in $S_L$ can be approximated with explicit forms. 
\subsection{Futile Cycle}
A futile cycle system  \eqref{eq:futile} appears in \cite{daigle2011automated,kuwahara2008efficient} as an example for computing rare event probabilities. In the system, species $S_2$ is transformed to $S_5$ through intermediate species $S_3$, and this transformation is catalyzed with $S_1$. In the opposite way, $S_4$ catalyzes the transform of $S_5$ to $S_2$ with the intermediate form $S_6$. We added synthesis and degradation of the catalysts $S_1$ and $S_6$ to the original model. 
\begin{align}
\begin{split}\label{eq:futile}
    &S_1+S_2\xrightarrow{\kappa_1} S_3, \quad S_3\xrightarrow{\kappa_2} S_1+S_2, \quad S_3\xrightarrow{\kappa_3} S_1+S_5\\
     &S_4+S_5\xrightarrow{\kappa_4} S_6, \quad S_6\xrightarrow{\kappa_5} S_4+S_5, \quad S_6\xrightarrow{\kappa_6} S_4+S_2,\\
     &\emptyset \xrightleftharpoons[\kappa_8]{\kappa_7} S_1, \quad \emptyset \xrightleftharpoons[\kappa_{10}]{\kappa_9} S_6.
     \end{split}
\end{align}

Let $N$ be the scaling parameter. We set the initial copies of the species as $S_2, S_3, S_5$ and $S_6$ have initially high copies and $S_1$ and $S_4$ have initially low copies.
In particular, $X^N_i(0)=N$ for $i=2,5,6$, $X^N_3(0)=2N$, $X^N_1(0)=2$ and $X^N_4(0)=1$. Hence $\S_L=\{S_1,S_4\}$ and $\S_H=\{S_2,S_3,S_5,S_6\}$. We choose the same rate constants as used in \cite{daigle2011automated,kuwahara2008efficient}, and we assume that the scaling parameter $\beta_k=0$ for all the reaction rate constants so that $\kappa_i=0.1$ for $i=3,6$ and $\kappa_i=1$ otherwise.

Under the mass-action kinetics, the initial reaction intensities are \\
$\lambda_1(X^N(0))=\kappa_1 X^N_1(0)X^N_2(0)=\kappa_1 2N$, $\lambda_2(X^N(0))=\kappa_2X^N_3(0)=\kappa_2 2N$, and so on.
Note that $\lambda_k(X^N(0))=\Theta(N)$ for each $y_k\to y'_k \in \Re_0$, so the maximum order of the initial reaction intensity is $\theta_0=1$.
By definition of $\Re_0$ \eqref{eq:decompose R}, the reactions are classified into $\Re_0$ and $\Re_0^c$ where $\Re^c_0$ contains the $7$ and $8$ th reactions, and the other reactions belong to $\Re_0$. 

We consider the scaled process $Z^{N,-\theta_0}(t)$ such that $Z^{N,-\theta_0})_i(t)=N^{-\alpha_i}X^N(N^{-\theta_0}t)$. We also consider the projected system $(\S_L,\C_L, \Re_L,\K_L)$.
By fix all the species $\S_H$ at their initial values, we obtain the parameter $s_k$ for the projected system such that
\begin{align*}
    s_1=\lim_{N\to \infty} \lambda^{N,-\theta_0}_{H,1}(Z^{N,-\theta_0}(0)) = \lim_{N\to \infty} \frac{N^{\beta_1+y_1 \cdot \alpha}}{N^{\theta_0}}\frac{X^N_2(0)}{N}=1,
\end{align*}
and so on. Thus as shown in Section \ref{subsec:projection} the projected system $(\S_L,\C_L, \Re_L,\K_L)$  is defined as
\begin{align}\label{eq:reduced futile}
     S_1 \xrightleftharpoons[2(\kappa_2+\kappa_3)]{\kappa_1}  \emptyset \xrightleftharpoons[\kappa_4]{\kappa_5+\kappa_6} S_4.
\end{align}
Both $S_1$ and $S_4$ in \eqref{eq:reduced futile} follow a simple birth-death process and hence the stationary distribution of $Z$ is a Poisson distribution so that the condition \eqref{condition3:tail} holds. Furthermore, unlikely in the original model \eqref{eq:futile}, probability densities of species $S_1$ and $S_5$ at time $t$ are analytically tractable as it is shown in \cite{jahnke2007solving} that the time evolution of the probabilities for a linear birth-death process is a convolution of Poisson distributions and  multinomial distributions. Thus we can analytically approximate the dynamics of the species in $S_L$ of the original system \eqref{eq:futile}. Figure \ref{fig3}A displays the density function of $S_1$ for $Z^{N,-\theta_0}(10)$ and the reduced system $Z(10)$ with $N=100$. Figure \ref{fig3}B displays the accuracy of the approximation indicating the convergence rate in the scaling parameter $N$ as proved in Theorem \ref{thm:main1}.

\subsection{Yeast Polarization}
In this section, we consider a signal-transduction pathway \eqref{eq:yeast}, which was introduced in \cite{daigle2011automated}. In the system, species $G$, so-called `G-proten', serves an important role in yeast polarization \cite{moore2013yeast}. $G$ goes through a separation-dephosphorylation-rebind cycle, as the 5th,6th and 7th reactions describe in \eqref{eq:yeast}, respectively. And this cycle is activated by ligand-receptor binding. 
\begin{align}
\label{eq:yeast}
\begin{split}
  &\emptyset \xrightleftharpoons[\kappa_2/N]{\kappa_1/N} R, \quad
    L+R\xrightarrow{\kappa_3/N}RL+L, \quad RL\xrightarrow{\kappa_4/N}R,,\\
    &RL+G \xrightarrow{\kappa_5}G_a+G_{bg} \quad G_a\xrightarrow{\kappa_6/N}G_d, \\
    &G_d+G_{bg}\xrightarrow{\kappa_7}G, \quad    \emptyset \xrightarrow{\kappa_8} RL.
    \end{split}
\end{align}

We model this yeast polarization system with a multiscale stochastic mass-action system.
Let $N$ be the scaling parameter. We suppose that the initial copies of ligand $L$, protein $G$ and its subunit $G_{bg}$ are $\Theta(1)$, and other species have the copy numbers of order $N$. More precisely, for a multiscale process $X^N(t)$ associated with \eqref{eq:yeast}, we set $X^N_R(0)=N,X^N_L(0)=2, X^N_{RL}(0)=N, X^N_G(0)=5, X^N_{G_\alpha}(0)=N, X^N_{G_{bg}}(0)=5$ and $X^N_{G_d}(0)=N$. 
As described in \eqref{eq:yeast}, we scale the rate constants (see the caption of Figure \ref{fig3} for the values of $\kappa_i$'s). Then by computing the reaction intensities at $X^N(0)$, we have the maximum order of reaction intensity $\theta_0$ and then we classify the reactions into $\Re_0=\{y_k\to y'_k:k=5,7\}$ and $\Re^c_0$.

We approximate the scaled process $Z^{N,-\theta_0}(t)$ under slow-timescale with the projected system $(\S_L,\C_L,\Re_L,\K_L)$. The parameters $s_k$ defined around \eqref{eq:rate constants for reduced} are
\begin{align*}
    s_5=\lim_{N\to \infty}\lambda^{N,-\theta_0}_{H,5}=\lim_{N\to \infty} N^{-\theta_0+\beta_5+y_5\cdot \alpha} Z^{N,-\theta_0}_{RL}(0)=1,\\
     s_7=\lim_{N\to \infty}\lambda^{N,-\theta_0}_{H,7}=\lim_{N\to \infty} N^{-\theta_0+\beta_7+y_7\cdot \alpha} Z^{N,-\theta_0}_{G_{d}}(0)=1.
\end{align*}
Then the projected system is
\begin{align}\label{eq:reduced yeast}
    G\xrightleftharpoons[\kappa_7]{\kappa_5} G_{bg}.
\end{align}
The stochastic process $Z(t)=(Z_G(t),Z_{G_{bg}}(t))$ associated with \eqref{eq:reduced yeast} admits a finite state space as the total quantity of $G$ and $G_{bg}$ is preserved. Hence \eqref{condition3:tail} hold. Furthermore the probability density function can be analytically derived as $p(t)=\mu e^{-Qt}$ where $Q$ is the transition matrix of $Z$ defined on $\{(z_1,z_2) : z_1+z_2=Z_G(0)+Z_{G_{bg}}(0)=10\}$ and $\mu$ is the initial distribution of $Z$ such that $\mu(5,5)=1$. We show model reduction in Figure \ref{fig3}CD with almost the same rate constants used in \cite{daigle2011automated}. In Figure \ref{fig3}C, letting $N=1000$, we compare the probability densities of $G_{bg}$ of $Z^{N,-\theta_0}$ associated with \eqref{eq:yeast} and its reduced system $Z$ associated with \eqref{eq:reduced yeast} at time $t=10$. Figure \ref{fig3}D furthermore shows the convergence rate of the approximation.

\begin{figure}[!htb]
\centering{
\includegraphics[width=1\textwidth]{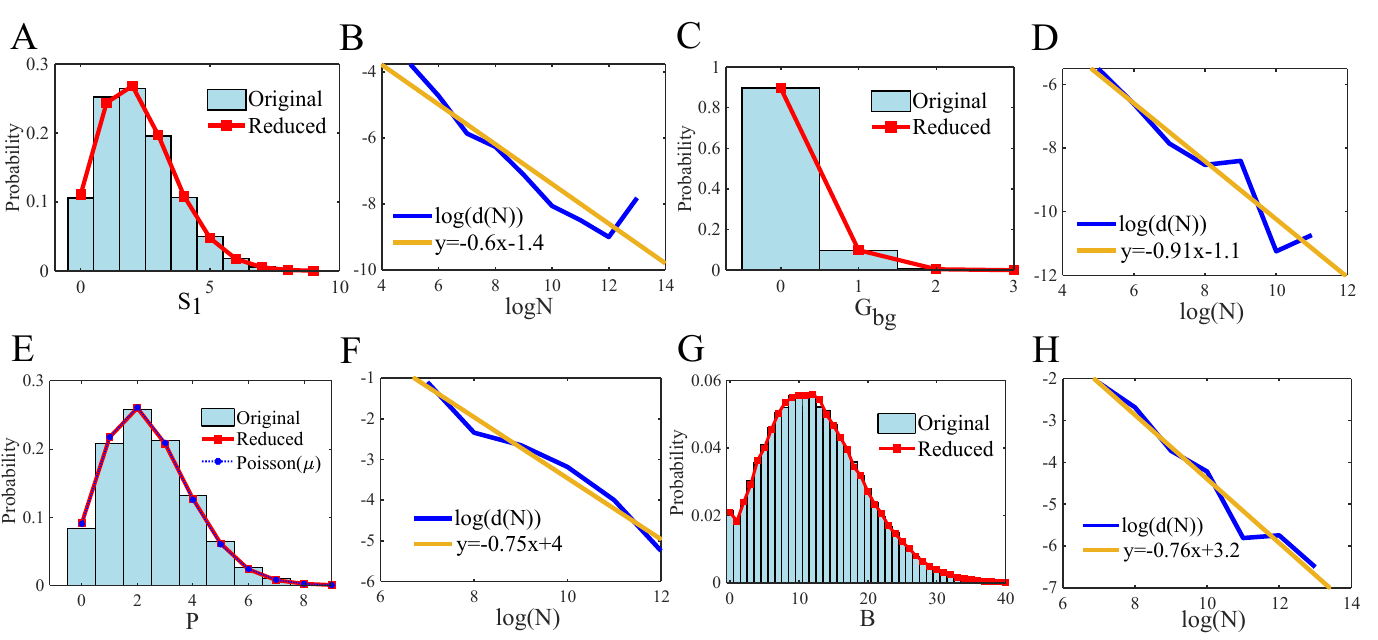}}
\caption{\textbf{AB.} Model reduction for the futile system \eqref{eq:futile}. Parameters are $\kappa_3=0.1$ and $\kappa_6=0.1$, and $\kappa_i=1$ otherwise. The probability density of the original model is calculated with $10^5$ times of Gillespie's simulations \cite{Gillespie77}. \textbf{A.} Comparison of the distributions $P(S^{N,-\theta_0}_1(t)=z_\ell))$ and $P(S_1(t)=z_\ell)$ at $t=100$ with $N=10^4$. \textbf{B.} Convergence $d(N)=|P(Z^{N,-\theta_0}(t)\in A)-P(Z(t) \in A)|$ as $N$ increases for $A=\{S_1=3 \text{ or } 4\}$ at $t=100$.
\textbf{CD.} Model reduction for the yeast polarization system \eqref{eq:yeast}. Parameters are $\kappa_1=3.8,
\kappa_2=40, \kappa_3=42, \kappa_4=10, \kappa_5=0.011, \kappa_6=10, \kappa_7=1$ and $\kappa_8=3.21$. The probability density of the original model is calculated with $10^7$ times of Gillespie's simulations.
\textbf{C.} Comparison of the distributions $P(G_{bg}^{N,-\theta_0}(t)=z_\ell))$ and $P(G_{bg}(t)=z_\ell)$ at $t=10$ with $N=10^3$. \textbf{D.} Convergence of $d(N)=|P(Z^{N,-\theta_0}(t)\in A)-P(Z(t) \in A)|$ as $N$ increases for $A=\{S_1=3 \text{ or } 4\}$ at $t=100$.
\textbf{EF.} Model reduction for the system of p53 response \eqref{eq:p53}. The parameters $\kappa_i$ are 
$1.1, 0.6, 0.3, 0.3, 3.4, 4.5, 4.1, 0.6, 1.1, 0.9, 1.9, 3.2, 3.2$, and $3.0$ for $i=1,2\dots,14$, respectively. Furthermore $c_1=4.7$ and $c_2=1.9$. Each parameter is sampled uniformly in $[0,5]$. $10^7$ Gillespie's simulations are used to estimate the probability densities of the original and reduced models. \textbf{E.}  Comparison of the distributions $P(P^{N,-\theta_0}_1(t)=z_\ell))$ and $P(P(t)=z_\ell)$ at $t=100$ with $N=10^5$. \textbf{F.} Convergence $d(N)=|P(P^{N,-\theta_0}(t)\in A)-P(P(t) \in A)|$ as $N$ increases for $A=\{P=3 \text{ or } 4\}$ at $t=100$. \textbf{GH.} Model reduction for the population model \eqref{eq:Lotka}. The parameters $\kappa_i$ are 
$0.5, 1.7, 3.9, 4.6, 2.7, 1.9, 6.1, 2.4$, and $1.5$ for $i=1,2\dots,9$, respectively. Each parameter is sampled uniformly in $[0,5]$. We sampled $10^5$ trajectories with Gillespie's algorithm to estimate the probability densities of the original and reduced models. \textbf{G.}  Comparison of the distributions $P(B^{N,-\theta_0}(t)=z_\ell))$ and $P(B(t)=z_\ell)$ at $t=150$ with $N=10^5$. \textbf{H.} Convergence $d(N)=|P(B^{N,-\theta_0}(t)\in U)-P(B(t) \in U)|$ as $N$ increases for $U=\{B\ge 10\}$ at $t=150$. The yellow straight lines in \textbf{BDFH} indicate that the convergence rate of $d(N)$ to $0$ is in $\Theta(\frac{1}{N^\nu})$ for some $\nu \in (0,1)$.}\label{fig3}
\end{figure}
%\hyperref[app:covergence rate]{B}
\subsection{p53 response to DNA damage}
Network \eqref{eq:p53} describes signaling pathway and negative feedback for activating p53 proteins in damages DNA, which is introduced in \cite{batchelor2008recurrent}.
When DNA is damages, signaling kinases (S) such as ATM and Chk2  convert inactive p53 protein ($P_0$) to active p53 protein (P). The p53 protein also represses itself by the negative feedback loop with the Mdm2 protein (M). Another negative feedback loop between p53, the signal and inhibitor (I) is also present in this system. See a schematic description of this system in \cite[Figure 1B]{batchelor2008recurrent}
\begin{align}
    \begin{split}\label{eq:p53}
        &P_0 + S \xrightarrow{N\kappa_1} P+S, \\
        &P_0+M \xrightarrow{\kappa_2} M, \quad   \ \ \ \ \qquad P+M \xrightarrow{\kappa_3} M, \\
        &P \xrightarrow{\kappa_4} P+M, \quad  \ \ \ \ \ \qquad  P\xrightarrow{\kappa_5} P +I,\\
        &S+M \xrightarrow{\kappa_6 /N}S, \quad  \ \ \qquad S+I \xrightarrow{\kappa_7}I,\\
        &P\xrightleftharpoons[N\kappa_9]{\kappa_8} \emptyset \xrightleftharpoons[\kappa_{11}]{\kappa_{10}} M, \quad  \qquad S \xrightleftharpoons[\kappa_{13}]{\kappa_{12}} \emptyset \xleftarrow{\kappa_{14}}I.
    \end{split}
\end{align}
To match the initial setting used in \cite{batchelor2008recurrent}, we assume that 
the inhibitor has low copies at the beginning. Furthermore we assume that $P_0$ and $P$ have also initially low copies. Precisely, $X^N_{P_0}(0)=5, X^N_{P_0}(0)=0, X^N_{I}(0)=1, X^N_{N}(0)=N$ and $X^N_{S}(0)=5N$. Hence $\S_L=\{P_0,P,I\}$ and $\S_H=\{M,S\}$.
For each state $x=(x_{P_0},x_{P},x_S,x_M,x_I)$, the reaction intensities $\lambda_1$ and $\lambda_7$ contain hill-functions as
\begin{align*}
    \lambda_1(x)=N\kappa_1 x_{P_0}\frac{x_S}{x_S+c_1}, \quad \text{and}\quad   \lambda_7(x)=\kappa_7 x_{S}\frac{x_I}{x_I+c_2},
\end{align*}
for some positive constants $c_1$ and $c_2$. The other intensities follow the mass-action kinetics \eqref{eq:mass}. Under the scaled rate constants shown in \eqref{eq:p53}. Then the maximum order $\theta_0$ of the intensities is $1$ so that $\Re_0$ contains reaction $1, 2, 3, 6, 7, 9, 11$ and $12$. For the process $Z^{N,-\theta_0}(t)$ under slow-timescale, the associated projected network $(\S_L,\C_L,\Re_L,\K_L)$ is
\begin{align*}
    & \quad P_0 \xrightarrow{ \ \quad  s_1\kappa_1   \quad \ } P \\
    &\text{\footnotesize{$s_8 \kappa_8$}}\searrow\hspace{-0.5cm}\nwarrow \text{\footnotesize{$\ s_2\kappa_2$}} \hspace{0.4cm} \swarrow\hspace{-0.2cm} \text{\footnotesize{$s_3\kappa_3$}}\\
    &\hspace{1.4cm} \emptyset
\end{align*}
 where the parameters $s_k$'s are defined as in \eqref{eq:rate constants for reduced}. Especially for the non-mass action intensity $\lambda^N_1(Z^{N,-\theta_0}(0))$ we have
\begin{align*}
   s_1=\lim_{N\to \infty} \lambda^{N,-\theta_0}_{H,1}(q_H(Z^{N,-\theta_0}(0))=N^{-\theta_0+\beta_1 }\frac{N}{N+c_1}=1.
\end{align*}
Note that these intensities under non-mass action kinetics satisfy \ref{condi1}--\ref{condi3}.
Let $Z(t)$ be the stochastic process associated with $(\S_L,\C_L,\Re_L,\K_L)$. The time evolution of the probability density of $Z(t)$ is analytically intractable. However, it can be shown that the probability density converges to a unique stationary distribution, which can be explicitly derived. $(\S_L,\C_L,\Re_L,\K_L)$ satisfies special network structure, so-called \emph{zero deficiency} and \emph{weakly reversible}, and hence its stationary distribution is a product form of Poisson distributions \cite[Theorem 6.1]{AndProdForm}. See Appendix for more details.

Figure \ref{fig3}E displays the probability distribution of the original and reduced models along with Poisson distribution of rate $\mu=2.4$, which is the average of $P$ in $(\S_L,\C_L,\R_L,\K_L)$. A commodity machine was used to simulate the samples in parallel (parfor in Matlab with 6 workers) and took 317 sec for the original model and 2.4 sec for the reduced model. We also show the convergence of the original model in Figure \ref{fig3}E.

\subsection{Three species Lotka-Volterra model with migration}

We consider a multiscale Lotka-Volterra population model \eqref{eq:Lotka} with the scailing parameter $N$. There exists three species in the network where $A$ is the lowest level prey, $B$ is the middle level species, and $C$ is the top level predator.
\begin{align}
\begin{split}\label{eq:Lotka}
    &B\xrightleftharpoons[\kappa_2]{\kappa_1} \emptyset, \quad A\xrightleftharpoons[\kappa_4]{\kappa_3/N} \emptyset, \quad C\xrightleftharpoons[\kappa_6 ]{\kappa_5 } \emptyset,\\
    & A \xrightarrow{\kappa_7/N} 2A, \quad A+B\xrightarrow{\kappa_8/N} 2B, \quad B+C\xrightarrow{\kappa_9} 2C.
    \end{split}
\end{align}
We use non-mass action kinetics for the reactions $A+B\to 2B$ and $B+C\to 2C$ to model `weak hunting' of $B$ and $C$ such that for each $x=(x_A,x_B,x_C)$,
\begin{align*}
    \lambda_8(x)=\frac{\kappa_8}{N} x_A\sqrt{x_B} \quad \text{and} \quad \lambda_9(x)=\kappa_9 \log(x_Bx_C+1).
\end{align*}
The other reaction intensities obey the mass-action kinetics. 
We assume that $\S_L=\{B,C\}$ and $\S_H=\{A\}$ as $X^N_A(0)=N$, $X^N_B(0)=10$ and $X^N_C(0)=10$. Under the scaled rate constants shown in \eqref{eq:Lotka}, the maximum order of the reaction intensities $\theta_0$ is $0$, and every reaction belongs to $\Re_0$. Note that we assumed that slow birth and slow degradation of $A$ so that the associated reaction rate constants for the 3rd and 7th reactions are of order $\frac{1}{N}$. 

For the scaled process $Z^{N,-\theta_0}(t)$, the scaled reaction intensities are decomposed as defined in \eqref{eq:decomp of lam}, especially for each $z=(z_A,z_B,z_C)$
$\lambda^{N,-\theta_0}_{H,8}=N^{\beta_8+y_8 \cdot \alpha}z_A, \lambda_{L,8}=\kappa_8 \sqrt{z_B}$, $\lambda^{N,-\theta_0}_{H,9}=1$ and $\lambda_{L,9}=\kappa_9\log(z_B z_C+1)$. Hence \ref{condi1}--\ref{condi3} hold.

Now we consider the projected system. As the the reaction $A+B\to 2B$ obeys non-mass action kinetics involving $A\in \S_H$, parameter the $s_8$ is especially computed as $s_8=\lim_{N\to \infty}\lambda^N_{H,8}(Z^{N,-\theta_0}(0))=1$. Hence the projected system is
\begin{align}\label{eq:reduced lotka}
     &B\xrightleftharpoons[\kappa_2]{\kappa_1} \emptyset, \quad C\xrightleftharpoons[\kappa_6 ]{\kappa_5 } \emptyset \quad B\xrightarrow{\kappa_8} 2B, \quad B+C\xrightarrow{\kappa_9} 2C.
\end{align}
Let $Z(t)$ be the stochastic process associated with \eqref{eq:reduced lotka}
In Appendix, we show how to use the Foster-Lyapunov criterion \cite{MT-LyaFosterIII} to verify that a stationary distribution $\pi$ of $Z$ exists and how $\pi$ meets the condition \eqref{condition3:tail}. The approximation and convergence rates of species $B$ is shown in Figure \ref{fig3}G and H. 
A commodity machine was used to approximate the probability densities in parallel (parfor in Matlab with 6 workers) and took 174 sec for the original model and 70 sec for the reduced model.

\section{Discussion}\label{sec:discussion}

When a stochastic biochemical reaction system contains species with different orders of abundance, one can model the system using a multiscaling approach. We have shown in this paper that a multiscale stochastic reaction a short-term timescale can be approximated using a reduced stochastic system with a specific error bound.

%Under other scaling regimes, the original process could converge to a purely deterministic process, a purely stochastic process, or a hybrid process.

The scaling regime we considered in this paper is a special case of the scaling under the so-called species balance condition, which was introduced in \cite{KangKurtz2013}. 
More general multiscaling limits of stochastic reactions introduced by Kurtz and others rely on the law of large numbers and relative compactness of probability measures in a metric spaces \cite{anderson2017finite, ball2006asymptotic, KangKurtz2013, Kurtz72}. This general framework covers a wide range of multiscaling limits, but the convergence rate in the general case remains unknown. Instead of the classical methods, here we have used a direct analysis of the Kolmogorov forward equation, and we also use the state space truncation through FSP to exploit the distance between two probability measures more explicitly. 

One of the key steps for the main result was to show that the concentrations of the order $N$ species are confined near the initial concentrations. To do this, we showed that the multiscale model is non-explosive by assuming that the reduced system admits a stationary distribution satisfying a finite moment condition. Indeed, this assumption implies the tightness of the family of multiscale stochastic processes, which in turn implies the relative compactness of the sequence of probability measures. This assumption is also closely related to some technical conditions on stoichiometric coefficients such as the binary or unary conditions assumed in \cite{anderson2017finite, KangKurtz2013}.

We can generalize the rate of the convergence in the main result if all the moments of the stationary distribution $\pi$ in \ref{condition3:tail} are finite. Based on a suggestion by Chaojie Yuan for this case, we used the Burkholder–Davis–Gundy inequality \cite{burkholder1972integral} and were able  to prove inductively the alternative result that $E(J(t)^m)\le c t^m$ in Lemma \ref{lem:mean of the jump number}. By combining this result with the other lemmas with slight modifications, it follows that
 \begin{align*}
 |p^N(A,\R^r_{\ge 0},t) - p(A,t)| \le  \frac{c \max\{1,t^m\}}{N^\nu} \quad \text{for any $\nu\in (0,1)$}.
 \end{align*}

The main result in this paper can shed light on the applicability of multiscaling model approximations for the analysis of stochastic reaction systems.  In the analysis that we proposed for the multiscale model reduction, the convergence of the probability measure has been exhibited more explicitly than in the existing literature.
The main result can also strengthen the applicability of this theoretical framework to practical problems in systems biology such as rational circuit design and the study of absolute robustness \cite{kim2020absolutely}.

\section*{Acknowledgment}

We would like to thank Eduardo Sontag, Carsten Wiuf, Chuang Xu and Linard Hoessly for key suggestions regarding this work, as well as Chaojie Yuan for an alternative proof of an important lemma.

%We can replace the condition $a_{11}\neq 0$ by all $a_{i1}'s$ are equal to zero.
%\begin{theorem}\label{thm:condition2}
%Let $(\S,\C,\Re)$ be a deficiency zero, weakly reversible reaction network with conservation laws \eqref{eq:conserve}. Let $S_1$ be the species of interest and suppose $a_{i1}= 0$ for all $i$.
% Suppose that the reduced network $(\widetilde{S},\widetilde{\C},\widetilde{\Re})$ obtained by freezing the species $S_{d+1},\dots,S_{d+r}$ satisfies either
%\begin{enumerate}
%\item all reaction vectors in $\widetilde{\Re}$ are independent to $(1,0,\dots,0)$, or
%\item both $S_1$ and $\emptyset$ is belonging to $\widetilde \C$.
%\end{enumerate}
%Then if we add a module of two reactions $Z+S_1\xrightarrow{a} 2Z$ and $Z\xrightarrow{b} S_1$ to $(\S,\C,\Re)$ and there exists a positive equilibrium for the deterministic model with the module, then the counts of $S_1$ is roughly Poisson-distributed with the rate $b/a$ as $ N$ goes to infinity. 
%\end{theorem} 
%\begin{proof}
%Note that we have an additional conservative laws $S_1(t)+Z(t)=\alpha_0 N$ by the module. So we can freeze $Z$. Then the rest of the proof is exactly same as the proof of Theorem \ref{thm:condition1}.
%\end{proof}
\section*{Appendix A: Table of symbols}\label{app:table}

\begin{center}
\vspace{0.2cm}
\begin{tabular}{|c|c|}
Symbol & Meaning \\
\hline
     $\S, \C, \Re$ and $\K$ & Set of species, complexes, reaction and reaction intensities, respectively\\
     $X_i(t)$ & The count of $i$ th species  at time $t$ \\
      $X^N(t)$ & A multiscale stochastic model associated with a reaction network \\
      $Z^{N,\gamma}(t)$ & A scaled process\\
      $d$ & Number of species of low initial copies\\
        $r$ & Number of species of high initial copies\\
      $\Z^d_{\ge 0}$ & $\{x \in \Z^d : x_i \ge 0 \text{ for each $i$} \}$\\
      $\Re^r_{\ge 0}$ &  $\{z \in \R^r : z_i \ge 0 \text{ for each $i$} \}$\\
      $p^N(\cdot,t)$ & Probability density function of $Z^{N,-\theta_0}(t)$\\
      $p(\cdot,t)$ & Probability density function of $Z(t)$\\
      $n^{(k)}$ & $n(n-1)\cdots (n-k+1)\mathbbm{1}_{n\ge k}$ for non-negative integers $n$ and $k$\\
      $u^{(v)}$ & $\prod_{i=1}^d u_i^{(v_i)}$ for $u,v \in \Z^d_{\ge 0}$\\
      $u^{v}$ & $\prod_{i=1}^d u_i^{v_i}$ for $u,v \in \R^d_{\ge 0}$\\
      $\S_L$ & Set of species such that $X_i(0)=\Theta(1)$\\
      $\S_H$ & Set of species such that $X_i(0)=\Theta(N)$\\
      $\lambda_k$ & Reaction intensity of $X^N$ associated with the $k$ th reaction.\\
      $\lambda^{N,\gamma}_k$ & Reaction intensity of $Z^{N,\gamma}$ associated with the $k$ th reaction \eqref{eq:scaled intensitiy}.\\
      $\lambda_{L,k}$, $\lambda^{N,\gamma}_{H,k}$ & Decomposition of $\lambda^{N,\gamma}_k$ \eqref{eq:decomp of lam}.\\
      ${y'}^N-y^N$ &  The scaled reaction vector with the $i$ th component $\frac{y'_{k,i}-y_{k,i}}{N^{\alpha_i}}$
\end{tabular}

\end{center}

\section*{Appendix B: Proof of Lemmas in Section \ref{sec:lemma}}\label{app1}

\begin{proof}[\textbf{Proof of Lemma \ref{lem:two intensities}}]
\blue{Suppose first that $y_k\to y'_k \in \Re_0$. As defined in \eqref{eq:decomp of lam}, we use the decomposition of $\lambda^{N,-\theta_0}_k(z)=\kappa_k \lambda_{L,k}(z_\ell)\lambda^{N,-\theta_0}_{H,k}(z_h)$.
As we discussed in Remark \ref{rmk:alternative K^N}, the intensities of $Z$ can be calculated as $\bar \lambda_u(z_\ell)=\bar \kappa_u (z_\ell)^{(\bar y_u)}=\bar \kappa_u \lambda_{L,k}(z_\ell)$ if $q_L(y_k)=\bar y_u$. Hence by definition of $S_M$ in \eqref{eq:compact sets},
\begin{align}
\begin{split}\label{eq:between lambdaN-lambda}
    \left | \lambda^{N,-\theta_0}_{k}(z)-\frac{s_k \kappa_k}{\bar \kappa_u}\bar \lambda_{u}(z_\ell)   \right |
    &=\kappa_k \lambda_{L,k}(z_\ell) \left | \lambda^{N,-\theta_0}_{H,k}(z_h)-s_k  \right | \\
    &\le \kappa_k \lambda_{L,k}\frac{M}{N} =\kappa_k \lambda_{L,k}\frac{1}{N^{1-\rho}},
    \end{split}
\end{align}
 for each $z\in S_M$. Therefore (i) follows with $\nu_1=1-\rho$.}

\blue{To show (ii), we recall that $-\theta_0+\beta_k+y_k \cdot \alpha<0$ for each $y_k\to y_k'\in \Re^c_0$, which implies that for $Z^{N,-\theta_0}(0)=z^0$
\begin{align*}
    s_k=\lim_{N\to \infty}N^{-\theta_0+\beta_k+y_k \cdot \alpha} \prod_{i=d+1}^{d+r}z^0_i\left(z^0_i-\frac{1}{N}\right )\cdots \left(z^0_i-\frac{y_{i}-1}{N}\right )=0.
\end{align*}
Thus (ii) follows with  $\nu_2=1-\rho( \max_k\Vert q_L(y_k)\Vert_\infty +1)$ by choosing sufficiently small $\rho\in (0,1)$ for $M=N^\rho$ because for $z\in S_M$
\begin{align*}
    \lambda^{N,-\theta_0}_k(z) \le \kappa_k\lambda_{L,k(z_\ell)}\frac{M}{N}\le  \kappa_k M^{\Vert q_L(y_k)\Vert_\infty } \frac{M}{N}=\kappa_k \frac{1}{N^{\nu_2}}.
\end{align*}
} 

\blue{Lastly, to show that (iii) we note that $\dsum_{y\to y' \in \Re_0}s_k\kappa_k \lambda_{L,k}(z)=\sum_u \kappa_u \bar \lambda_u(z)$. %Furthermore \eqref{eq:min bar lambda} implies that for each constant $a,b>0$, we have $\dfrac{a}{N^{b}}\le \displaystyle \min_{z_\ell} \sum_u \bar \lambda_u(z_\ell)$ for sufficiently large $N$. 
Then by (i) and (ii) there exists $c>0$ such that for any $z \in S_M$
\begin{align*}
    \sum_k\lambda^{N,-\theta_0}_k(z)&=\sum_{y_k\to y_k' \in \Re_0} \lambda^{N,-\theta_0}_k(z)+\sum_{y_k\to y_k' \in \Re^c_0} \lambda^{N,-\theta_0}_k(z)\\
    &\le \sum_{y_k\to y_k' \in \Re_0} \left | \lambda^{N,-\theta_0}_{k}(z)-\frac{s_k \kappa_k}{\bar \kappa_u}\bar \lambda_{u}(z_\ell)   \right | +  \sum_{y_k\to y_k' \in \Re_0} \frac{s_k \kappa_k}{\bar \kappa_u}\bar \lambda_{u}(z_\ell)\\
    & \ \ \ +\sum_{y_k\to y_k' \in \Re^c_0} \lambda^{N,-\theta_0}_k(z)\\
    &\le \sum_{y_k\to y_k' \in \Re_0}\frac{\kappa_k\lambda_{L,k}(z_\ell)}{N^{\nu_1}}+\sum_{y_k\to y_k' \in \Re_0}\frac{s_k\kappa_k}{\bar \kappa_u}\bar \lambda_u(z_\ell)+\frac{c}{N^{\nu_2}} \\
    &\le \frac{c}{N^{\nu_1}}\sum_u \bar \lambda_u(z_\ell) + \sum_u \bar \lambda_u(z_\ell)+\frac{c}{N^{\nu_2}}= \left ( 1+\frac{c}{N^{\nu_1}}\right)\sum_u \bar \lambda_u(z_\ell)+\frac{c}{N^{\nu_2}}.
\end{align*}
Then by \eqref{eq:min bar lambda}, the upper bound of $\sum_k \lambda^{N,-\theta_0}_k(z)$ follows. The lower bound also holds as
\begin{align*}
     \sum_k\lambda^{N,-\theta_0}_k(z) &\ge \sum_{y_k\to y_k' \in \Re_0} \lambda^{N,-\theta_0}_k(z) \\
      &\ge \sum_{y_k\to y_k' \in \Re_0} \left ( \lambda^{N,-\theta_0}_{k}(z)-\frac{s_k \kappa_k}{\bar \kappa_u}\bar \lambda_{u}(z_\ell)   \right ) + \sum_{y_k\to y_k' \in \Re_0} \frac{s_k \kappa_k}{\bar \kappa_u}\bar \lambda_{u}(z_\ell)\\
      &\ge -\sum_{y_k\to y_k' \in \Re_0}\frac{\kappa_k\lambda_{L,k}(z_\ell)}{N^{\nu_1}} + \sum_{y_k\to y_k' \in \Re_0}\frac{s_k \kappa_k}{\bar \kappa_u}\bar \lambda_u(z_\ell)\\
    &\ge  \left ( -\frac{c}{N^{\nu_1}}+1\right)\sum_u \bar \lambda_u(z_\ell)
\end{align*}
}
\end{proof}

\begin{proof}[\textbf{Proof of Lemma \ref{lem:mean of the jump number}}]
 By the random-time representation \eqref{eq:kurtz rep}, \\
 $J(t)=\dsum_{\bar y_u\to \bar y'_u \in \Re_L}Y_u\left ( \dint_0^t \bar \lambda_u(Z(s)) ds\right )$, where $Y_u$ are independent unit Poisson random variables. Note that  for each $u$,
\begin{align*}
Y_u\left ( \int_0^t \bar \lambda_u(Z(s)) ds\right )-\int_0^t \bar \lambda_u(Z(s)) ds 
\end{align*} is a Martingale process \cite{AndKurtz2011}. We denote by $M_u(t)$ this Martingale. Then  the quadratic variation of $M_u(t)$ is $[M_u](t)=Y_u\left ( \int_0^t \bar \lambda_u(Z(s)) ds\right )$.
Since $M^2_u(t)-[M_u](t)$ is a martingale \cite{AndKurtz2011}, we have that by using Jansen's inequality
\begin{align}
E\left (Y_u\left ( \int_0^t \bar \lambda_u(Z(s)) ds\right )^2\right ) &\le 2E(M_u(t)^{2})+  2E\left( \left( \int_0^t \bar \lambda_u(Z(s)) ds\right)^2 \right ) \notag \\
&\le 2E([M_u](t))+2tE\left(\int_0^t \bar \lambda_u(Z(s))^2 ds\right)  \label{eq:J^2}
\end{align}

For a fixed initial value $Z(0)=z^0$, there exists $c>0$ such that $P(Z(s)=z)\le c\pi(z)$ for any $z$ because
\begin{align*}
   \pi(z)=\sum_{x}P(Z(s)=z \ | \ Z(0)=x)\pi(x) \ge P(Z(s)=z)\pi(z^0).
\end{align*}
This implies that by \ref{condition3:tail}, we have 
\begin{align*}
E([M_u](t))&=E\left( Y_u\left ( \int_0^t \bar \lambda_u(Z(s)) \right ) \right ) = \int_0^t \sum_{x}\sum_{u}\bar \lambda_u(x) P(Z(s)=x) ds\\
&\le c_1\sum_{x}\bar \lambda_u(x) \pi(x)t \le c'_1 t,
\end{align*}
with some positive constants $c_1$ and $c'_1$.
Similarly, 
\begin{align*}
E\left(\int_0^t \bar \lambda_u(Z(s))^2 ds\right)  &= \int_0^t\sum_{x}\bar \lambda_u(x)^2 P(X(s)=x)ds \\
&\le c'_2\int_0^t\sum_{x}\bar \lambda_u(x)^2 \pi(x)ds \le c_2 t,
\end{align*}
with some positive constants $c_2$ and $c'_2$. Applying these to \eqref{eq:J^2}, it follows that
\begin{align*}
E\left (Y_u\left ( \int_0^t \bar \lambda_u(Z(s)) ds\right )^2\right ) \le c'\max\{1,t^2\},
\end{align*}
with some positive constant $c'$.

 Finally the result follows since  by Jansen's inequality we have that there exists a positive constant $c$ such that
\begin{align*}
E(J(t)^2)\le |\Re_L|\sum_{u} E\left (\left (Y_u\left ( \int_0^t \bar \lambda_u(Z(s)) ds\right )\right )^2 \right ) \le c\max\{1,t^2\},
\end{align*}
\end{proof}

\begin{proof}[\textbf{Proof of Lemma \ref{lem:number of jumps needed}}]
It is suffice to show that at least $cM$ transitions are required for $Z^{N,-\theta_0}(t)$ to escape $S_M$. Let $Z^{N,-\theta_0}(0)=z^0$ be the initial state. We first show that $\{z_h : |z_h-z^0_h| < c\frac{M}{N} \}\subseteq S_{H,M}$ for some $c>0$. Note that if $y_k\to y_k' \in \Re_0$, then $s_k=(z^0_h)^{q_H(y_k)}$ and  $-\theta_0+\beta_k+y_k\cdot \alpha=0$. Then there exist $c'>0$ and $c''>0$ such that for any $y_k\to y_k' \in \Re_0$ if $z_h=z^0_h+\eta$ with $|\eta|\le c'\frac{M}{N}$, then 
\begin{align*}
    &|z_h^{q_H(y_k)}-s_k|=|(z^0_h+\eta)^{q_H(y_k)}-s_k| \le c''|\eta|\le \frac{M}{2N}, \quad \text{and}\\
    &\left |z_h^{q_H(y_k)}-\lambda^{N,-\theta_0}_{H,k}(z_h) \right | =\left |\prod_{i=1}^r z_{h,i}^{y_{k,d+i}}-\prod_{i=1}^{r} z_{h,i}\left( z_{h,i}-\frac{1}{N}\right)\cdots\left( z_{h,i}-\frac{y_{k,d+i}-1}{N}\right) \right |   \\
    & \hspace{3.5cm}\le \frac{c''}{N}.
\end{align*}
This implies that for $y_k\to y'_k\in \Re_0$ if $|z_h-z^0_h| \le c'\frac{M}{N}$ then
\begin{align*}
    \left |\lambda^{N,-\theta_0}_{H,k}(z_h)-s_k\right | \le \left |z_h^{q_H(y_k)}-s_k\right |+
    \left |z_h^{q_H(y_k)}-\lambda^{N,-\theta_0}_{H,k}(z_h)\right | \le \frac{M}{N},
\end{align*}
for sufficiently large $N$.
Hence $\{z_h : |z_h-z^0_h| < c\frac{M}{N} \}\subseteq S_{H,M}$.
Furthermore $\{z_\ell : |z_\ell - z^0_\ell| \le c'''M\} \subseteq S_{L,M}$ for some $c'''>0$ when $N$ is sufficiently large. Therefore
\begin{align*}
    \{z_\ell : |z_\ell - z^0_\ell| \le c'M\}\times 
    \left \{z_h : |z_h-z^0_h| < c'''\frac{M}{N} \right \} \subset S_M.
\end{align*}
Since the transition size for each entry $Z^{N,-\theta_0}_i$ of the scaled process transitions by a single reaction is $\Theta(N^{-\alpha_i})$, this implies that $Z^{N,-\theta_0}(t)$ needs at least $\lfloor cM \rfloor$ transitions for some $c>0$ to escape $S_M$. 
\end{proof}

\begin{proof}[\textbf{Proof of Lemma \ref{lem:Gronwall}}]
$\frac{d}{dt}v(t)\le Av(t)+b$  can be written as
\begin{align*}
v(t)\le tb+ \int_0^t Av(s) ds
\end{align*}
allowing that the inequality holds component-wisely. Then by the multivariable Gronwall's inequality \cite{MultiGronwall},
\begin{align*}
v(t)\le tb+t\int_s^t V(t,s)Abds,
\end{align*}
where $V(t,s)$ satisfies 
\begin{align*}
V(t,s)=I+\int_s^t AV(x,s) dx,
\end{align*}
allowing the equality holds component-wisely.
By taking time-derivative, we notice that $i$ th column of $V(t,s)$ is a solution $u$ of the system of differential equation \eqref{eq:system ode} with $u_i(s)= 1$ and $u_j(s)=0$ if $j \neq i$. 
Therefore by the hypothesis in the statement, each column of $V(t,s)$ is a positive vector and the sum of the entries is equal to $1$ for $t\ge s$. Hence it implies that $V(t,s)Ab\le \bar Ab$ for $t\ge s$. Thus for each $t$, the result follows.
\end{proof}

\end{document}